\documentclass[a4paper,12pt]{amsart}
\pdfoutput=1 % must stay in the first five lines of the preamble for ArXiv

\usepackage{amsmath,amssymb,verbatim, enumerate,mathrsfs}

\usepackage{tikz, pgffor, pgf}
\usepackage[colorlinks]{hyperref}

\usepackage{times}

\def\Nodes[#1][#2]{
 \foreach \x in {1,...,#2} {
	 \filldraw [black] (\x, #1) circle (2pt);
	 }
	}
\def\ShiftX{\pgftransformxshift}

\def\XUNIT{1} 	\def\YUNIT{1}
\newcommand{\VCenter}[2]{
\begin{tikzpicture}[xscale=\XUNIT, yscale=\YUNIT]
		\draw[white] (0,0.5*#1)--(0,-0.5*#1);
		\node at (0,0) {#2};
\end{tikzpicture}
}

\input xy
\xyoption{all}

\usepackage[enableskew]{youngtab}

\usepackage{url}

\newdimen\hoogte    \hoogte=12pt    % hoogte  van hokje
\newdimen\breedte   \breedte=14pt   % breedte van hokje
\newdimen\dikte  \dikte=0.5pt    % dikte lijn
\def\beginYoung{
 \begingroup
 \def\vr{\vrule height0.8\hoogte width\dikte depth 0.2\hoogte}
 \def\fbox##1{\vbox{\offinterlineskip
    \hrule height\dikte
    \hbox to \breedte{\vr\hfill##1\hfill\vr}
    \hrule height\dikte}}
 \vbox\bgroup \offinterlineskip \tabskip=-\dikte \lineskip=-\dikte
  \halign\bgroup &\fbox{##\unskip}\unskip  \crcr }
\def\End@Young{\egroup\egroup\endgroup}

\newtheorem{theorem}{Theorem}[section]
\newtheorem{proposition}[theorem]{Proposition}
\newtheorem{lemma}[theorem]{Lemma}
\newtheorem{corollary}[theorem]{Corollary}

\newtheorem{example}[theorem]{Example}

\theoremstyle{definition}
\newtheorem*{defi}{Definition}

\newtheorem*{remark}{Remark}

\newcommand{\Z}{\mathbb{Z}}    %Z integer numbers
\newcommand{\C}{\mathbb{C}}    %complex numbers
\newcommand{\short}[1]{\hbox{$#1$}}

 \def\vv{\mathrm{v}} \def\ww{\mathrm{w}} \def\ii{\mathrm{i}}  \def\cB{\mathcal{B}}\def\cD{\mathcal{D}}\def\cG{\mathcal{G}}\def\cP{\mathcal{P}}

\def\End{\mathrm{End}} 
\def\id{\mathrm{id}}  

\title{The Quasi-Partition Algebra}
\date{\today}

\author{Zajj Daugherty}
 \thanks{Z.\ Daugherty is partially supported by
    NSF Grant DMS-1162010}

\author{Rosa Orellana}   

\address{Zajj Daugherty and Rosa Orellana, Dartmouth College, Mathematics Department, 6188 Kemeny Hall, Hanover, NH 03755, USA.}    
\email{Zajj.B.Daugherty@dartmouth.edu}
\email{Rosa.C.Orellana@dartmouth.edu}

\begin{document}
\maketitle
\begin{abstract}
We introduce the quasi-partition algebra $QP_k(n)$ as a centralizer algebra of the symmetric group.  This algebra is a subalgebra of the partition algebra and inherits many similar combinatorial properties. We construct a basis for $QP_k(n)$, give a formula for its dimension in terms of the Bell numbers, and describe a set of generators for $QP_k(n)$ as a complex algebra. In addition, we  give the dimensions and indexing set of its irreducible representations. We also provide the Bratteli diagram for the tower of quasi-partition algebras (constructed by letting $k$ range over the positive integers).
\end{abstract}
%\tableofcontents

\section*{Introduction}

We introduce the centralizer algebra $QP_k(n)$, the {\em quasi-partition algebra}.  This algebra arises as a subalgebra of the partition algebra, $P_k(n)$, which was introduced independently by Jones \cite{Jones}  and Martin \cite{Martin} as a generalization of the Temperley-Lieb algebra and the Potts model in statistical mechanics.   Jones defined $P_k(n)$ as a centralizer algebra and explicitly described the Schur-Weyl duality between $P_k(n)$ and the symmetric group $S_n$.  Specifically, $P_k(n)$ generically centralizes the action of the symmetric group on the $k$-fold tensor product of the permutation representation $V$ of $S_n$, i.e.
$$
P_k(n) \cong \End_{S_n}(V^{\otimes k}) \qquad \text{ when } n\geq 2k.
$$
The partition algebra has a basis indexed by set partitions, these set partitions can be encoded into graphs that make the partition algebra into a diagram algebra with multiplication given by concatenation of diagrams.

The permutation representation $V$ decomposes into a direct sum of the trivial representation $S^{(n)}$ and the irreducible reflection representation $W = S^{(n-1,1)}$.
We define $QP_k(n)$ as the centralizer 
$$QP_k(n) = \End_{S_n} (W^{\otimes k}).$$
We describe a basis for $QP_k(n)$, which is indexed by set partitions of $2k$ elements without sets of size one. The dimension is therefore the number of such partitions, given by a formula in terms of the Bell numbers.  

From our construction $QP_k(n)$ is  a subalgebra of $P_k(n)$, in addition we show that $QP_k(n)$ is isomorphic to a subalgebra of $P_k(n-1)$, and exploit this relationship to provide a set of generators for $QP_k(n)$ and find relations satisfied by these generators.  We give a formula for the product in $QP_k(n)$ and show that  it is dominated by the relations in $P_k(n-1)$,

Using the rule for decomposing the tensor product, $W\otimes S^\lambda$ of $W$ with any other irreducible representation $S^\lambda$ of the symmetric group,  we show that for $k\geq 2$ the irreducible representations of $QP_k(n)$ are indexed by the set of partitions of $0,1, 2, \ldots, k$.    We have constructed the Bratteli diagram, which encodes inclusion and restriction rules between $QP_{k-1}(n) $ and  $QP_{k}(n)$.  We also give a formula for the dimensions of the irreducible representations for $QP_{k}(n)$. 

We expect some results in the representation theory of $QP_k(n)$ to correspond to results about the Kronecker product. One should be able, through  a further exploration of the structure of this new algebra, to obtain results about symmetric functions and the representation theory of the symmetric group.

\section{The Partition Algebra}
The structure of the quasi-partition algebra $QP_k(n)$ is understood through the structure of the partition algebra $P_k(n)$. A basis for each algebra is encoded both as set partitions and as diagrams, and actions of both algebras on tensor space are calculated from those diagrams. Combinatorial results about the irreducible representations of $QP_k(n)$ will resemble those for the partition algebra as well. In this section we set the stage by describing the general partition algebra $P_k(x)$ in terms of the partition diagrams and describing its action on tensor space.

\subsection{Set partitions and partition diagrams} \label{section:partalgdiag} 
A \emph{set partition} of a set $S$ is  a set of pairwise disjoint subsets of $S$, called \emph{blocks}, whose union is $S$. Fix $k \in \Z_{>0}$, and denote
$$[k] = \{1, \dots, k\} \qquad \text{ and } \qquad [k'] = \{1', \dots, k'\},$$
so that $[k] \cup [k'] = \{1, \dots, k, 1', \dots, k'\}$ is formally a set with $2k$ elements.

For each set partition of $[k] \cup [k'] $, we associate a {diagram} as follows. Consider the set of  simple graphs with $2k$ vertices labeled from $[k] \cup [k']$, and draw the graph so that the vertices appear in two rows, $1, \dots k$ on the top and $1', \dots, k'$ on the bottom. Any two vertices in the same block of the set partitions are connected by a path. In particular, the connected components of the graph correspond to the blocks in the set partition.  Define two graphs to be equivalent if their connected components partition the $2k$ (labeled) vertices in the same way. Then we define a \emph{$k$-partition diagram} or simply \emph{diagram} as the equivalence class of graphs corresponding to the same set partition of $[k]\cup[k']$.  For example,

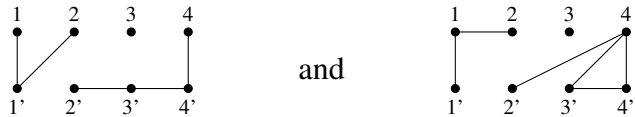
\begin{center}
	{\def\XUNIT{.75} \def\YUNIT{.75}
	\begin{tikzpicture}[scale=\YUNIT]
		\foreach \x in {1,...,4} {
			\node[above] at (\x,1) {\tiny \x};
			\node[below] at (\x,0) {\tiny \x'};
			}
		\Nodes[0][4]	\Nodes[1][4]
		\draw (1,1)--(1,0)--(2,1) (2,0)--(3,0)--(4,0)--(4,1);
	\end{tikzpicture}
	\qquad \VCenter{2}{and} \qquad 
	\begin{tikzpicture}[scale=\YUNIT]
		\foreach \x in {1,...,4} {
			\node[above] at (\x,1) {\tiny \x};
			\node[below] at (\x,0) {\tiny \x'};
			}
		\Nodes[0][4]
		\Nodes[1][4]
		\draw (1,0)--(1,1)--(2,1) (4,1)--(3,0)--(4,0)--(4,1)--(2,0);
	\end{tikzpicture}
}\end{center}
are equivalent, and both represent diagrams for the set partition $\{ \{1,2,2'\}, \{3\}, \{2', 3', 4', 4\}\}$.

Let $\C(x)$ be the field of rational functions with complex coefficients in an indeterminate $x$.  We define the product $d_1 \cdot d_2$ of two diagrams $d_1$ and $d_2$ using the concatenation of $d_1$ above $d_2$, where we identify
the southern vertices of $d_1$ with the northern vertices of $d_2$.   
If there are $c$ connected components consisting only of  middle vertices, then the product is set equal to $x^c$ times the diagram  with the middle components removed. Extending this linearly defines a multiplication on $P_k(x)$.

\noindent For example,
\begin{center}
	{\def\UNIT{.75}
	\begin{tikzpicture}[scale=\UNIT]
		\Nodes[0][4]	\Nodes[1][4]
		\draw (1,1)--(2,1) (1,0)--(2,0) (3,0)--(4,0)--(4,1);
		%%%%%%%
		\ShiftX{4cm};
		\node at (1,.5) {$*$};
		%%%%%%%
		\ShiftX{1cm};
		\Nodes[0][4]
		\Nodes[1][4]
		\draw  (3,1)-- (1,0) (4,1)--(4,0)--(3,0)--(2,0);
		%%%%%%%
		\ShiftX{4cm};
		\node[right] at (1,.5) {$= \quad x\,*$};
		%%%%%%%
		\ShiftX{3cm};
		\Nodes[0][4]
		\Nodes[1][4]
		\draw (1,1)--(2,1) (4,1)--(4,0)--(3,0)--(2,0)--(1,0);
	\end{tikzpicture}
}\end{center}

\noindent This product is associative and independent of the graph chosen to represent the partition diagram.

The \emph{partition algebra} $P_k(x)$ is the $\mathbb{C}(x)$-span of the $k$-partition diagrams with this product (with $P_0(x) = \mathbb{C}(x)$).  Under this product, $P_k(x)$ is an associative algebra with identity given by the diagram corresponding to $\{\{1,1'\}, \dots, \{k,k'\}\}$.  The dimension of  $P_k(x)$ is the number of set partitions of $2k$ elements, i.e. the \emph{Bell number} $B(2k)$.

\subsection{Generators and Relations of $P_k(x)$}\label{sec:gensP} A presentation for $P_k(n)$ has been given in \cite{HalversonRam} and in \cite{East}.  
Let 
$$\begin{tikzpicture}[xscale=.4, yscale=.5]
	\node[left] at (-2,.5) {$b_i = $};
	 \foreach \x in {-2,0,1,...,3,5} {
	 \filldraw [black] (\x, 0) circle (2pt);
	 \filldraw [black] (\x, 1) circle (2pt);
	 }
	\draw (1,0)--(1,1)--(2,1)-- (2,0)--(1,0);
	\draw  (0,1)--(0,0) (-2,1)--(-2,0) (3,1)--(3,0) (5,1)--(5,0);
	\node at (-1,.5) {$\dots$};
	\node at (4,.4) {$\dots$};
	\node at (1,1.5) {$i$};
	\end{tikzpicture} 
\qquad	
\begin{tikzpicture}[xscale=.4, yscale=.5]
	\node[left] at (-2,.5) {$p_i = $};
	 \foreach \x in {-2,0,1,2,4} {
	 \filldraw [black] (\x, 0) circle (2pt);
	 \filldraw [black] (\x, 1) circle (2pt);
	 }
	\draw  (0,1)--(0,0) (-2,1)--(-2,0) (2,1)--(2,0) (4,1)--(4,0);
	\node at (-1,.5) {$\dots$};
	\node at (3,.3) {$\dots$};
	\node at (1,1.5) {$i$};
	\end{tikzpicture} 
\qquad \text{and}\qquad \begin{tikzpicture}[xscale=.4, yscale=.5]
	\node[left] at (-2,.5) {$s_i = $};
	 \foreach \x in {-2,0,1,...,3,5} {
	 \filldraw [black] (\x, 0) circle (2pt);
	 \filldraw [black] (\x, 1) circle (2pt);
	 }
	\draw (1,0)--(2,1)  (1,1)--(2,0);
	\draw  (0,1)--(0,0) (-2,1)--(-2,0) (3,1)--(3,0) (5,1)--(5,0);
	\node at (-1,.5) {$\dots$};
	\node at (4,.4) {$\dots$};
	\node at (1,1.5) {$i$};
	\node[right] at (5,.5){.};
	\end{tikzpicture} 
$$	

\begin{theorem}[Theorem 1.11 of \cite{HalversonRam}] Fix $k\in\Z_{>0}$. The partition algebra $P_k(x)$ is the unital associative  
$\mathbb{C}$-algebra presented by the generators
$b_i$, $s_i$, and $p_j$ for $1\leq i\leq k-1$ and $1\leq j\leq k$,
together with Coxeter, idempotent, commutation and contraction relations.
\end{theorem}

\medskip

It is often useful to additionally distinguish the element 
$$\begin{tikzpicture}[scale=.5]
	\node at (-5,.5) {$e_i =b_ip_ip_{i+1}b_i=\ \ $};
	 \foreach \x in {-2,0,1,...,3,5} {
	 \filldraw [black] (\x, 0) circle (2pt);
	 \filldraw [black] (\x, 1) circle (2pt);
	 }
	\draw (1,0) (1,1)--(2,1) (2,0)--(1,0);
	\draw  (0,1)--(0,0)  (-2,1)--(-2,0) (3,1)--(3,0) (5,1)--(5,0);
	\node at (-1,.5) {$\dots$};
	\node at (4,.4) {$\dots$};
	\node at (1,1.5) {$i$};
	\end{tikzpicture} 
$$
for $1\leq i \leq k-1$. Using subsets of $\{s_i, e_i~|~ 1 \leq i \leq k-1\}$, one can generate the Temperley-Lieb algebra $TL_k(x)$, the group algebra of the symmetric group $\C S_k$, and the Brauer algebras $B_k(x)$ all as subalgebras of the partition algebra $P_k(x)$.

%%%PARTITION ALGEBRA AS CENTRALIZER ALGEBRA
\subsection{$P_k(n)$ as a centralizer algebra of $S_n$.}
\label{section:centralizer}
Let $V$ denote the $n$-dimensional permutation representation of the symmetric group $S_n$.  That is, $V=\C\mbox{-span}\{v_i\,|\, 1\leq i\leq n\}$, where 
\begin{equation}
\label{Vaction}
\sigma\cdot v_i = v_{\sigma(i)}\qquad \mbox{ for } \sigma \in S_n.
\end{equation}
Let $S_n$ act diagonally on the basis of simple tensors in $V^{\otimes k}$:
\[\sigma\cdot v_{i_1}\otimes v_{i_2} \otimes \cdots \otimes v_{i_k} = v_{\sigma(i_1)}\otimes v_{\sigma(i_2)} \otimes \cdots \otimes v_{\sigma(i_k)},\]
and extend this action linearly to $V^{\otimes k}$.  Hence $V^{\otimes k}$ is a module for $S_n$.

As above, number the vertices of a $k$-partition diagram $1,\ldots, k$ from left to right in the top row and $1', \ldots, k'$ from left to right on the bottom row.  For each $k$-partition diagram $d$ and each integer sequence $i_1\ldots, i_k, i_{1'}, \ldots, i_{k'}$ with $1\leq i_r\leq n$, define
\begin{equation}
\label{eq:diagrammatrix}
\delta(d)_{i_{1'}, \ldots, i_{k'}}^{i_1,\ldots, i_k} = \begin{cases} 1 & \mbox{ if $i_t= i_s$ whenever vertices $t$ and $s$ are connected in $d$,} \\ 0 & \mbox{ otherwise.}
\end{cases}
\end{equation}
Define an action of a partition diagram $d\in P_k(n)$ on $V^{\otimes k}$ by defining it on the standard basis by 
\[d\cdot (v_{i_{1}} \otimes v_{i_{2}} \otimes \cdots \otimes v_{i_{k}}) = \sum_{1 \leq i_{1'}, \ldots i_{k'}\leq n} \delta(d)_{i_{1'}, \ldots, i_{k'}}^{i_1,\ldots, i_k} v_{i_{1'}} \otimes v_{i_{2'}} \otimes \cdots \otimes v_{i_{k'}}.\]
Thus, the actions of the preferred generators of $P_2(n)$ on $V^{\otimes 2}$ are given by
\begin{equation}\label{actionsofP}
\begin{array}{r@{~}l}
\phantom{\Big|}	b\cdot (v_i\otimes v_j) = \delta_{ij} v_i\otimes v_i,  & \qquad e \cdot (v_i\otimes v_j) =  \sum\limits_{\ell=1}^n v_\ell \otimes v_\ell,\\
\phantom{\Big|} s \cdot (v_i\otimes v_j) = v_j \otimes v_i, \qquad &  \text{ and } \qquad  p\otimes \mathrm{id} \cdot (v_i\otimes v_j) = \left(\sum\limits_{\ell=1}^n v_\ell\right)\otimes v_j,
\end{array}\end{equation}
where 
\begin{equation}\label{generator-diagrams}
{\def\UNIT{.5}
\begin{tikzpicture}[scale=\UNIT]
	\node[left] at (1,.5) {$\phantom{bpes}$$b =$ ~};
	 \foreach \x in {1,2} {
		 \filldraw [black] (\x, 0) circle (2pt);
		 \filldraw [black] (\x, 1) circle (2pt);
	 }
	\draw (1,0)--(1,1)--(2,1)-- (2,0)--(1,0);
	\node[right] at (2,.5){~,$\phantom{bpes}$};
	\end{tikzpicture} 	
\begin{tikzpicture}[scale=\UNIT]
	\node[left] at (1,.5) {$\phantom{bpes}$$e =$ ~};
	 \foreach \x in {1,2} {
		 \filldraw [black] (\x, 0) circle (2pt);
		 \filldraw [black] (\x, 1) circle (2pt);
	 }
	\draw (1,1)--(2,1)  (2,0)--(1,0);
	\node[right] at (2,.5){~,$\phantom{bpes}$};
	\end{tikzpicture} 	
\begin{tikzpicture}[scale=\UNIT]
	\node[left] at (1,.5) {$\phantom{bpes}$$s =$ ~};
	 \foreach \x in {1,2} {
		 \filldraw [black] (\x, 0) circle (2pt);
		 \filldraw [black] (\x, 1) circle (2pt);
	 }
	\draw (1,0)--(2,1)  (1,1)--(2,0);
	\node[right] at (2,.5){~,$\phantom{bpes}$and};
	\end{tikzpicture} 	
\begin{tikzpicture}[scale=\UNIT]
	\node[left] at (1,.5) {$\phantom{bpes}$$p =$ ~};
	 \foreach \x in {1} {
		 \filldraw [black] (\x, 0) circle (2pt);
		 \filldraw [black] (\x, 1) circle (2pt);
	 }
	\node[right] at (1,.5){~.$\phantom{bpes}$};
	\end{tikzpicture}} 
\end{equation}
Then $s_i, e_i, b_i$, for $i=1, \ldots, k-1$ can be identified with the maps in $\End(V^{\otimes k})$ given by
$$\id^{\otimes i-1} \otimes d \otimes \id^{\otimes k-i+1}$$
where $d$ is one of $s, e$ or $b$; the elements $p_i$ for $i=1, \ldots, k$  can be identified with the maps in $\End(V^{\otimes k})$ given by
$$ \id^{\otimes i-1}\otimes p \otimes \id^{\otimes k-i}.$$

\begin{theorem}[\cite{Jones}]\label{thm:PkFullCentralizer} $S_n$ and $P_k(n)$ generate full centralizers of each other in $\mbox{End}(V^{\otimes k})$.  In particular, %for $n\geq 2k$, 
\begin{enumerate}[\quad(a)]
\item $P_k(n)$ generates $\mbox{End}_{S_n}(V^{\otimes k})$, and when $n\geq 2k$,  $P_k(n) \cong \mbox{End}_{S_n}(V^{\otimes k})$;
\item $S_n$ generates $\mbox{End}_{P_k(n)}(V^{\otimes k})$.
\end{enumerate}

\end{theorem} 
For the remainder of the paper, we use Theorem \ref{thm:PkFullCentralizer} to identify the elements of $P_k(n)$  (with $n \geq 2k$ integers) with endomorphisms of $V^{\otimes k}$.

%%%%%%%%%%%%%%%%%--- SECTION 2 -- %%%%%%%%%%%%%%%%%%%%%%%%%%%%%%%%

\section{The Quasi-Partition Algebra}
In this section we are interested in studying the centralizer algebra 
\[QP_k(n)=\mbox{End}_{S_n}(W^{\otimes k}), \qquad \text{ where } \quad W = S^{(n-1, 1)}\]
is the irreducible representation of $S_n$ indexed by the partition $(n-1,1)$.  With $V$ as in \eqref{Vaction}, it is known that $V$ decomposes as  $V= T\oplus W$, where $T$ is the trivial representation (indexed by the partition $(n)$).
\subsection{Action of $S_n$ on $W=S^{(n-1,1)}$}
Using the same basis $\{v_1, \ldots, v_n\}$ for $V$ as above, we fix a basis $\{w_2, \dots, w_{n-1}\}$ for $W$, where  $w_i = v_i-v_1$. The permutation action of $S_n$ on $V$ in  \eqref{Vaction} induces an action of $S_n$ on $W$ given by 
\[ \sigma \cdot w_i = w_{\sigma(i)}, \qquad \mbox{for $\sigma\in S_{\{2, \ldots, n\}}$} \qquad \mbox{ and } \qquad s_1\cdot w_i = \begin{cases}  w_i-w_2 & \mbox{for $i\neq 2$}\\ -w_2 & \mbox{for $i=2$}.\end{cases}\]
With $S_n$ acting diagonally on $W^{\otimes k}$, we define the \emph{Quasi-partition algebra} as the centralizer algebra 
\[ QP_k(n) =\End_{S_n} (W^{\otimes k}) = \{ g: W^{\otimes k} \rightarrow W^{\otimes k} ~|~ g\sigma=\sigma g \quad  \forall \sigma\in S_n\}.\]

The partition algebra $P_k(n-1)$ can be recognized as a subalgebra of $\End(W^{\otimes k})$ via the following change of basis. Define 
\[
f: \{v_1, \ldots, v_{n-1} \} \rightarrow \{w_2, \ldots, w_n\}  \qquad \mbox{ by  } \quad f: v_i \mapsto w_{i+1},
\]
and extend linearly. Then if $d$ is a diagram in $P_k(n-1)$, set 
\begin{equation}\label{eq:BracketD}
	[d] = f\circ d \circ f^{-1}.
\end{equation}  Hence, the action of $[d]$ on $W^{\otimes k}$ is the same as the action of $d$ on $V^{\otimes k}$ in one fewer dimension, i.e. $n$ has decreased by 1.   
Specifically, since the map $d \mapsto [d]$ is a homomorphism, we have the property that if $d_1, d_2\in P_k(n-1)$, then $[d_1] [d_2] = [d_1d_2]$. We remark that the maps $[d]  : W^{\otimes k}\rightarrow W^{\otimes k}$ are not necessarily elements in the centralizer algebra $QP_k(n)$. We will, however, use them to construct the elements of $QP_k(n)$.

\subsection{Projections}

Since $V=W\oplus T$, we have \\
\centerline{$\displaystyle 
V^{\otimes k} \cong W^{\otimes k} \oplus \left(\bigoplus_{i=0}^{k-1} (W^{\otimes i} \otimes T^{k-i})\right).
$}
In this section, we construct the maps in $QP_k(n)$ by applying maps in $P_k(n)$ on $W^{\otimes k}$ and then projecting the result back onto $W^{\otimes k}$.

To this end, first define the projection 
$\varpi: V\to T $
to be a projection of $V$ onto $T$.  Hence, 
\[\varpi ( v_i) = \short{\frac{1}{n}}( v_1 + \cdots  + v_n)\qquad \mbox{ for all $i=1, \ldots, n$}\]
The matrix representation of  $n\varpi$ is the matrix of all $1$'s, i.e.\  $\sum_{1\leq i, j, \leq n} E_{ij},$ 
were $E_{ij}$ are the matrix units. 
Now define 
$\varpi_\ell = \mathbf{1}^{\otimes \ell-1} \otimes \varpi \otimes \mathbf{1}^{\otimes k-\ell}.$
The endomorphism ring $\mbox{End}(V^{\otimes k})$ has basis $E_{j_1, \ldots, j_k}^{i_1, \ldots, i_k}$, and as a matrix $\varpi_\ell$ is given by
$$n\varpi_\ell 
	= \hspace{-.2cm}
	\sum_{{1\leq i,j \leq n} \atop {1 \leq a_m \leq n \text{ for } m \neq \ell}} 
	\hspace{-.2cm} E^{a_1, \dots, a_{\ell - 1}, i, a_{\ell+1}, \dots, a_{k}}_{a_1, \dots, a_{\ell - 1}, j, a_{\ell+1}, \dots, a_{k}}.$$
So as an operator on $V^{\otimes k}$, $n\varpi_\ell = p_\ell$, the element of $P_k(n)$ with isolated vertices at $\ell$ and $\ell'$, and all other blocks of the form $\{i, i'\}$ for $i\neq \ell$ (see diagram at the beginning of Section \ref{sec:gensP}).  We will show that a basis of  $QP_k(n)$ is given by diagrams that do not have isolated vertices.  As we will see in Lemma \ref{elements}, this is a result of the fact that that isolated vertices occur when $d=p_\ell d'$ or $d=d'p_\ell$ for some diagram $d'$ and some $1\leq \ell \leq k$. 

Now we can define the projection $\pi : V\rightarrow W$  by 
\[\pi = \mbox{id} - \varpi, \qquad \text{ so that } \qquad \pi^{\otimes k} = \pi\otimes \pi \otimes \cdots \otimes \pi  \]
projects  $V^{\otimes k}$ onto $W^{\otimes k}$. 
To simplify computations, we transform bases of $V$ from $\{v_1, \ldots, v_n\}$  to  $\{v, w_2, \ldots, w_n\}$, where 
\[v=\sum_{i=1}^n v_i \qquad \mbox{ and} \qquad w_j=v_j-v_1.\]
So 
$$\pi(v) = 0, \quad \pi(w_i) = w_i, \quad \text{and} \quad \pi(v_i) = w_i -\short{\frac{1}{n}}w, \qquad \text{where} \quad w_1 = 0 \quad \text{and} \quad  w = \sum_{i=2}^nw_i.$$

\begin{lemma}\label{elements}
For all diagrams $d\in P_k(n)$, the projection $\pi^{\otimes k}\circ d$ is an element of  $QP_k(n)$.   Furthermore,  if $d$ is a diagram with one or more isolated vertices, then $\pi^{\otimes k} \circ d = 0$. 
\end{lemma}

\begin{proof}  Since $\pi = \id - \frac{1}{n}p$ is an operator on $V$  (with $p$ as in \eqref{generator-diagrams}), $\pi$ commutes with the action of $S_n$. So $\pi^{\otimes k}\circ d \in \End_{S_n}(V^{\otimes k}) = P_k(n)$. 
Now considering $W^{\otimes k} \subset V^{\otimes k}$, we have
$$\pi^{\otimes k}\circ d : W^{\otimes k} \xrightarrow{d} V^{\otimes k} \xrightarrow{\pi^{\otimes k}} W^{\otimes k}.$$ 
So $\pi^{\otimes k}\circ d$ is also an element of $\End_{S_n}(W^{\otimes k})=QP_k(n)$. 

Now suppose that $d$ is a diagram with an isolated vertex. If the isolated vertex occurs in the bottom row of the diagram on the $i$-th vertex, then $d= d' p_i$ for some diagram $d'$. But 
$p = n\varpi$ as operators on $V$, so $p$ acts as 0 on $W$.   So $\pi^{\otimes k} \circ d \cdot w_{i_1} \otimes \cdots \otimes w_{i_k} =0$. 

If instead, the isolated vertex occurs in the top of the diagram on the $i$-th vertex, then $d = p_i d'$. So again since  $p=n\varpi$ as operators on $V$, we have 
$$\pi^{\otimes k} (p_i d') = n\left(\pi^{\otimes i-1} \otimes (\pi \circ \varpi) \otimes \pi^{\otimes (k-i)}\right) d' = 0, \qquad \text{ since } \pi \circ \varpi =0.$$
\end{proof}

%BASIS
\subsection{Basis for the Quasi-partition algebra} In Lemma \ref{elements} we found a spanning set for $QP_k(n)$,  we now show that  the set of projections of the diagrams without singleton vertices will form a basis.  That is, 
$$QP_k(n) =  \C\text{-span}\{ \bar{d} \, |\,  d \in \cD \}, \qquad \text{ where } \quad  \cD = \{ \text{ diagrams $d$ without isolated vertices}\},$$
and $\bar d = \pi^{\otimes k} \circ d$.
In order to prove this, we will show that if $d$ is a diagram without isolated vertices, then one can write $\pi^{\otimes k} \circ d$ as a linear combination of the operator $[d]$ (as defined in \eqref{eq:BracketD}) and operators $[d']$ with isolated vertices.

Recall from \eqref{eq:BracketD} that for any diagram $d$, the operator $[d] = f \circ d \circ f^{-1} \in \End(W^{\otimes k})$ acts on $W^{\otimes k}$ the same way that 
the diagram $d$ acts on $(\C^{n-1})^{\otimes k}$.  Notice that $[d]$ is an element of $\End_{S_{\{2, \dots, k-1\}}}(W^{\otimes k})$.  However, $[d]$ is in general not an element of $\End_{S_n}(W^{\otimes k})$.

Considering $d$ as a set partition, let $B_i$ denote the blocks of the diagram $d$ and $|d|$ be the number of blocks in the diagram, so that  $d = \{ B_1, \dots, B_{|d|}\}.$
Let $B$ be a block in $d$, and define 
\[ B^t := B \cap [k] \quad \text{ and } \quad B^b := B \cap [k'],\]
so that $B^t$ (resp.\ $B^b$) is the set of vertices in $B$ which are on the top (resp.\ bottom) of the diagram. 
Then 
\begin{equation} \label{eq:diagramActionCondidion}
d \cdot (v_{i_{1'}} \otimes \cdots \otimes v_{i_{k'}}) = \begin{cases} v_{i_1} \otimes \cdots \otimes v_{i_k} & \mbox{if  for each $B\in d$,}\\
&\mbox{  $i_\ell = i_m$  for all  $\ell, m \in B \subseteq [k]\cup[k']$}, \\
0 & \mbox{otherwise}.
\end{cases}
\end{equation}

Now, let 
$$d^t = \{ B \in d ~|~ B^b = \emptyset \} \qquad \text{ and } \qquad d^b = \{ B \in d ~|~ B^t = \emptyset \}$$
be the sets of blocks containing vertices only on the top or bottom of the diagram, respectively. 
If $X$ is a subset of $[k]\cup[k']$, define the \emph{isolation of $d$ (at $X$)} as 
\begin{equation*}
d_X , \text{ the diagram constructed from $d$ by isolating all vertices in $X$.}
\end{equation*}
For example, if $k=4$, $X=\{ 1', 4'\}$ and $d =\{ \{1,1',2'\}, \{2,3,4\}, \{3',4'\}\}$, then \\
$d_X=\{ \{ 1,2'\}, \{1'\}, \{2,3,4\}, \{3'\}, \{4'\}\}$.  In pictures, 

$$
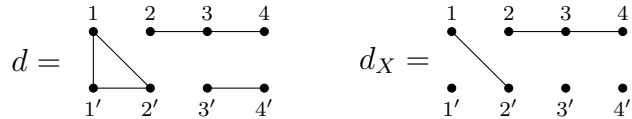
\begin{tikzpicture}[scale=.75]
	\Nodes[1][4];
	\Nodes[0][4];
	\node at (0,.5) {$d =$};
	\draw (1,0)--(1,1)--(2,0)--(1,0) ;
	\draw (2,1)--(3,1)--(4,1) (4,0)--(3,0) ;
	\foreach \x in {1, ..., 4}{
	\node[above] at (\x, 1) {\tiny $\x$};
	\node[below] at (\x, 0) {\tiny $\x'$};
	}
	\end{tikzpicture}
\qquad 
\begin{tikzpicture}[scale=.75]
	\Nodes[1][4];
	\Nodes[0][4];
	\node at (0,.5) {$d_X =$};
	\draw (1,1)--(2,0)  (1,0);
	\draw (2,1)--(3,1)--(4,1); %(4,0)--(3,0)  ;
	\foreach \x in {1, ..., 4}{
	\node[above] at (\x, 1) {\tiny $\x$};
	\node[below] at (\x, 0) {\tiny $\x'$};
	}
	\end{tikzpicture}	
	$$
Notice that two different sets $X_1$ and $X_2$ can lead to the same isolation of $d$. In the above example, the set $X_2=\{1', 3', 4'\}$ and $X_1=\{1', 4'\}$ give $d_{X_2} = d_{X_1}$.

\begin{lemma}\label{lem:DominantTerms1}
The action of $\bar d:=\pi^{\otimes k} \circ d$ on $W^{\otimes k}$ is equal to the action of a linear combination of $[d]$ and diagrams $[d']$, where $d'$ is an isolation of $d$. That is, 
\begin{equation}\label{eq:DominantTerms1}
\bar d = [d] + \sum_{U } c_U [d_U].\end{equation}
where the sum is over non-empty subsets $U$ of vertices satisfying 
$$\text{ if \quad $U \cap B^b \neq \emptyset$ \quad for any block $B \in d$, \quad then $B \subseteq U$.}$$

\end{lemma}
\begin{proof}
The first step is to understand how $d$ acts on an arbitrary element $\ww\in W^{\otimes k}$, 
$$\ww= w_{i_{1'}} \otimes \cdots \otimes w_{i_{k'}} = (v_{i_{1'}} - v_1) \otimes \cdots \otimes (v_{i_{k'}} - v_1).$$
Note that $\ww$ is the sum of terms $\vv$ with factors $v_{i_\ell}$s or $-v_1$s.  The collection of blocks $B$ in $d\setminus d^t$ (blocks containing bottom vertices) checks  for equality of factors in $\vv$ corresponding to vertices in those blocks.  Since $d\cdot \vv$ is zero if factors corresponding to vertices in the same block are not equal.  Since $i_\ell \neq 1$, there are only two types of terms when $d$ acts by non-zero: (1) on terms $\vv$ where  where for each block $B$, for all $\ell \in B^b$, the $i_\ell$'s take on the same value and these factors are all equal $v_{i_\ell}$ and (2) on terms where all factors corresponding to vertices in $B^b$ are all $-v_1$.  

We now carry the computation $(\pi^{\otimes k} \circ d )\cdot \ww$ .  Let  $S \subseteq d \setminus d^t$, and for 
 $\ii = (i_1, \dots, i_k, i_{1'}, \dots, i_{k'})$ let $\delta_{S, \ii}$ be the characteristic function
$$\delta_{ S, \ii} = \begin{cases}
1 & \text{ if for each $B \notin S$,  $i_\ell = i_m$ for all $\ell, m \in B$,  and
}\\
& \text{ for each $B \in S$,  $i_\ell = 1$ for all $\ell \in B^t$,}\\
0& \text{ otherwise.}
\end{cases}$$

Setting $w_1 = 0$, 
\begin{align}\label{eq:bigSsum}
(\pi^{\otimes k}\circ d) \cdot \ww
 &= \pi^{\otimes k}\cdot \sum_{S \subseteq d \setminus d^t} (-1)^{\sum_{B \in S} |B^b|} \sum_{i_1, \dots, i_k \in [n]} \delta_{S, \ii} (v_{i_1} \otimes \cdots \otimes v_{i_k})\\
 &= \sum_{S \subseteq d \setminus d^t} (-1)^{\sum_{B \in S} |B^b|} \sum_{i_1, \dots, i_k \in [n]} \delta_{S, \ii} ((w_{i_1} - \short{\frac1n}w) \otimes \cdots \otimes (w_{i_k} - \short{\frac1n}w)).\nonumber\end{align}
If $d^t = \emptyset$, then this implies that as operators on $\ww$
 \begin{align*} 
 \pi^{\otimes k}\circ d &= \sum_{X = \bigcup_{B \in S} B \atop S \subseteq d \setminus d^t} (-1)^{|X \cap [k']|}
	 \sum_{Y\subseteq [k] \setminus X}\left(\short{-\frac{1}{n}}\right)^{|Y\cup (X\cap[k])|}[d_{X \cup Y} ]\\
	&=\sum_{X, Y \atop U = X \cup Y} (-1)^{|U|}\frac1{n^{|U \cap [k]|}}[d_{U}],
\end{align*}
where the last sum is a double sum over sets $X$ and $Y$ of vertices satisfying 
\begin{equation}\label{eq:doubleUsum}
X = \bigcup_{B \in S} B \quad \text{ with } S \subseteq d \setminus d^t, \qquad 
\text{and } \qquad  Y \subseteq [k] \setminus X.\end{equation}
However, if  $B=\{j_1, \dots, j_r\} \in d^t$,   then isolating the factors in positions $j_i, \dots, j_r$ in $(\pi^{\otimes k} \circ d) \cdot \ww$ yields 
\begin{align*}
\pi^{\otimes k} \cdot \sum_{\ell = 1}^n v_\ell^{\otimes r} 
	&= (\short{\frac{-1}{n}})^r w^{\otimes r} + \sum_{\ell = 2}^n (w_\ell - \short{\frac1n}w)^{\otimes r}\\
	&= n(\short{\frac{-1}{n}})^r w^{\otimes r} 
		+ \sum_{a=0}^{r-1} \sum_{\sigma\in S_r/(S_a\times S_{r-a})} \sigma \cdot \left((\short{\frac{-1}{n}})^a w^{\otimes a}\otimes\left(\sum_{\ell=2}^n w_\ell^{\otimes (r-a)}\right)\right),
	\end{align*}
	where $\sigma \in S_r/(S_a\times S_{r-a})$ acts by permuting the $r$ factors, stabilizing the relative positions of the factors of $w^{\otimes a}$ and the the factors of $w_\ell^{\otimes (r-a)}$.

So the term $w^{\otimes r}$ appears with an extra factor of $n$. Thus, in general
 \begin{equation} 
 \pi^{\otimes k}\circ d = \sum_{X, Y \atop U = X \cup Y} (-1)^{|U|}\frac1{n^{|U \cap [k]| - n(Y)}}[d_{U}],
\end{equation}
where $$n(Y) = |\{ B \in d^t ~|~ B \subseteq Y\}|.$$
\end{proof}

As noticed before Lemma \ref{lem:DominantTerms1},  a diagram $d_U$ may arise non-uniquely as a function of $U$. As before, an isolation of a   diagram $d$ is a refinement $\hat{d}$ gotten from $d$ by isolating vertices. If those vertices satisfy the condition\\
\centerline{if any vertex of $B^b$ is isolated, then all vertices of $B$ are isolated,}
we call $\hat{d}$ a \emph{viable isolation}.

\begin{lemma}\label{lem:DominantTerms2}
Suppose $\hat{d}$ is a viable isolation of $d$ with set of isolated vertices $U$, and let $c$ be the coefficient of $[\hat{d}]$ after collecting like-terms in \eqref{eq:DominantTerms1}.  Then $c$ falls into the following cases: 
\begin{enumerate}
\item Non-unique terms, i.e. $c$ is the sum of multiple terms:
\begin{enumerate}
\item If $B \in d$ has exactly one vertex in $[k']$, and the vertices of $B$ are isolated in $\hat{d}$ , then $c = 0$.
\item Assume there is no block as in (a) completely isolated in $\hat d$, and that $d^t$ is nonempty. If $\cB = \{B_1, \dots, B_\ell\} \subseteq d^t$, and all blocks in $\cB$ but no other blocks of $d^t$ are isolated in $\hat{d}$, then 
$$c= (-1)^{|U|}\left( \frac1{n^{|U \cap [k]| - \ell}}\right)\left( \sum_{\mathcal{B}' \subseteq \{B_1, \dots, B_\ell\}} \left((-1)^{|\mathcal{B}'|} \prod_{B \in \mathcal{B}'}|B|\right)\right)$$
 where $\prod_{B \in\cB'}|B| = 1$ when $\cB' = \emptyset$. 

\end{enumerate}
\item Unique terms: If $\hat{d}$ is not one of the cases in (1), then $\hat{d}$ appears uniquely with coefficient $c=(-1)^{|U|}\frac{1}{n^{|U \cap [k]|}}$ where $U$ is the set of isolated vertices of $\hat{d}$.  In particular, if the isolated vertices of $\hat{d}$ are all in $[k']$, then then $c=(-1)^{\#\{\text{isolated vertices}\}}$
\end{enumerate}
If $U\cap [k]\neq \emptyset$, then  $\lim_{n \to \infty} c= 0$.
\end{lemma}

\begin{proof}
Use the same notation as in Lemma \ref{lem:DominantTerms1}, and continue from its proof. 

In the first case (1) note that a diagram $[d_U]$ can appear multiple times, if selecting different sets $X$ and $Y$ breaks up an entire block in different ways. This can happen in two cases. 
\begin{enumerate}
\item[(a)] \emph{ If some $B \in d\setminus d^t$ has the property that $B^b$ has a single element, then including $B$ in $X$ or including $B^t$ in $Y$ result in the same $d_U$. }

In this case, pair each $(X, Y)$ such that $B \subset X$ with $(X', Y')$ given by 
$$X' = X \setminus \{B\}, \quad \text{ and } Y' = Y\cup B^t $$
so that $[d_{X \cup Y}]=[d_{X' \cup Y'}]$. These terms occur with equal but opposite coefficients, since 
$$|(X \cup Y) \cap [k]| = |(X' \cup Y')\cap [k]|  \qquad \text{ but } \qquad |X' \cup Y'| = |X \cup Y|-1,$$
and so cancel.

\item[(b)] \emph{ If $B \in d^t$, then the diagrams with $B \subseteq Y$ are exactly equal to those with one vertex of $B$ removed from $Y$ ($Y' = Y \setminus \{j\}$ for some $j \in B$).} 

Fix a collection of blocks  $B_1, \dots, B_\ell$ in $d^t$, and let $U$ be some set of the form $U = X \cup Y$ as in \eqref{eq:doubleUsum}, such that $\bigcup_{1 \leq i \leq \ell} B_i \subseteq U$, but $d_U$ has no other blocks in $d^t$ completely broken up. Then $[d_U]$ will appear with coefficient 
\begin{align*}
	c = (-1)^{|U|}\left( \frac1{n^{|U \cap [k]| - \ell}}\right)  &\qquad \text{ for the set $U$},\\
	\left((-1)^{|\mathcal{B}'|} \prod_{B \in \cB'}|B|\right)c& \qquad\text{ for each non-empty}  \cB' \subseteq \{B_1, \dots, B_\ell\}
\end{align*}
where the second value counts the cases $U= X \cup Y'$, where $Y'$ removes one element from each $B \in \mathcal{B}$ from $Y$:  for each $B$, there are $|B|$ ways to accomplish this, and each resulting $Y'$ has $n(Y') = n(Y) -|\cB'|$, $|U'| = |U| - |\cB'|$, and $|U' \cap [k]| = |U \cap [k]| - |\cB'|$. So, collecting like-terms, the coefficient on $[d_U]$ is 
$$c_U = (-1)^{|U|}\left( \frac1{n^{|U \cap [k]| - b}}\right)\left( \sum_{\cB'\subseteq \{B_1, \dots, B_\ell\}} \left((-1)^{|\cB'|} \prod_{B \in\cB'}|B|\right)\right) ,$$
 where $\prod_{B \in\cB'}|B| = 1$ when $\cB' = \emptyset$.

\end{enumerate} 
If $U \subseteq [k'] $, the term $[d_U]$ appears exactly once; the isolated vertices are unions of blocks in $d^b$, and so coefficient is $(-1)^{|U|}$. 
Otherwise, $[d_U]$ appears exactly once, with coefficient $(-1)^{|U|}\frac{1}{n^{|U \cap [k]|}}$. In general if $U\cap[k]\neq \emptyset$, then $\lim_{n \to \infty} c_U = 0$. 
\end{proof}

\begin{example}\label{ex:DominantTerms} Let $d$ be the diagram 
$$
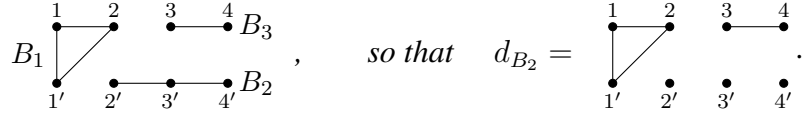
\begin{tikzpicture}[scale=.75]
	\Nodes[1][4];
	\Nodes[0][4];
	\draw (1,0)--(1,1)--(2,1)-- (1,0);
	\draw (3,1)--(4,1) (4,0)--(3,0)--(2,0) ;
	\node[left] at (1, .5) {\small $B_1$};
	\node[right] at (4, 1) {\small $B_3$};
	\node[right] at (4, 0) {\small $B_2$};
	\foreach \x in {1, ..., 4}{
	\node at (\x, 1.3) {\tiny $\x$};
	\node at (\x, -.3) {\tiny $\x'$};
	}
	\end{tikzpicture} 
	\begin{tikzpicture}[scale=.75]
	\node[left] at (.5, .5) {,\qquad so that \quad $d_{B_2} = $};
	\Nodes[1][4];
	\Nodes[0][4];
	\draw (1,0)--(1,1)--(2,1)-- (1,0);
	\draw (3,1)--(4,1); 
	\foreach \x in {1, ..., 4}{
	\node at (\x, 1.3) {\tiny $\x$};
	\node at (\x, -.3) {\tiny $\x'$};
	}
	\node at (4.3,.47) {.};
	\end{tikzpicture}$$
For the three blocks labeled $B_1, B_2$, and $B_3$,
	$$B_1^t = \{1,  2\}, \quad  B_1^b = \{1'\}, \quad  B_2^t = \emptyset, \quad  B_2^b = \{2',   3',   4'\}, \quad  B_3^t = \{3,  4\}, \quad  \text{ and } B_3^b = \emptyset.$$ 
So $d^t = \{B_3\}$, $d^b = \{B_2\}$, and the sum in \eqref{eq:bigSsum} is over the sets 
$$S = \emptyset, \{B_1\},  \{B_2\}, \text{ or }  \{B_1, B_2\}.$$ 
The summand where $S = \{B_1, B_2\}$ is 
\begin{equation}\label{eq:domTermEx1}
(-1)^{3+3} \sum_{i \in [n]} v_1 \otimes v_1 \otimes v_i \otimes v_i,
\end{equation}
which projects via $\pi^{\otimes k}$ to 
\begin{equation}\label{eq:domTermEx2}
\frac{1}{n^2} w \otimes w \otimes \left( \sum_{i =2, \dots, n} (w_i - \frac{1}{n}w) \otimes (w_i - \frac{1}{n}w) \right).\end{equation}
The corresponding summands when $S = \{B_1\}$ are $-\delta_{i_2, i_3} \delta_{i_3, i_4}$ times the expressions in \eqref{eq:domTermEx1} and \eqref{eq:domTermEx2}. 
The terms where $U \subseteq [k']$ are when $S = \emptyset$ or $\{B_2\}$, and expand to 
$$(-1)^0 [d] + (-1)^3 [d_{B_2}].$$

\end{example}

\begin{theorem}
A basis for $QP_k(n)$ is given by 
\[ \{ \bar{d} \, |\, d \in \cD\}  \qquad \text{ where } \quad  \cD = \{ \text{ diagrams $d$ without isolated vertices}\}\]
as before.
\end{theorem}

\begin{proof} 
We have already established $\{ \bar{d} \, |\, d \in \cD\}$ as a spanning set.  It remains to show that it is also linearly independent. 
 By Lemma \ref{lem:DominantTerms1}, for each diagram $d$ without singletons, there is exactly one element $\bar{d}$ which has $[d]$ with non-zero coefficient the expansion of its projection (namely, $d$ itself). Since the diagrams form a basis of $P_k(n-1)$ when $n-1 \geq 2k$, and the elements $[d] \in \End(W^{\otimes k})$ generate an isomorphic subalgebra when $n-1 \geq 2k$, this implies that $\{ \bar{d} \, |\, d \in \cD\} $ 
is linearly independent.
\end{proof}

We obtain as easy consequences of this theorem both dimension formulas and a result concerning diagram multiplication.
\begin{corollary}
 If $n\geq 2k+1$, then the dimension of $QP_k(n)$ is the number of set partitions of $2k$ without blocks of size one.  This number is given by
  \[\sum\limits_{j=1}^{2k} (-1)^{j-1}B(2k-j)+1, \qquad \text{ where $B(r)$ is the Bell number.}\]
\end{corollary}

This can be easily checked by noticing that the number of set partitions without singletons, $a(r)$, satisfies the recurrence
$a(r+1) + a(r) = B(r)$ subject to $a(1) =0$ and $a(2)=1$, and then that the formula satisfies this recurrence.

\begin{corollary}\label{cor:barProduct}
Given $d_1$ and $d_2$ any two diagrams without singleton vertices, then 
\[ \bar{d_1}\bar{d_2} = \sum_{d \in \cD \atop d\leq  d_1d_2} c_{d_1, d_2}^{d}\bar d,\]
where $d\leq d'$ if every block of $d'$ is the union of blocks of $d$, i.e.\ $d$ is a {\em refinement} of $d'$.
 \end{corollary}

\begin{proof}
{If $d_1, d_2 \in \cD$, and $d_1', d_2'$ are isolations of $d_1, d_2$, then $d_1'd_2' \leq d_1 d_2$.} By Lemma \ref{lem:DominantTerms1} 
$$\bar d_1 = [d_1] + \sum_{U} a_U [(d_1)_U] \qquad \text{ and } \qquad \bar d_2 = [d_2] + \sum_{V} b_V [(d_2)_V]$$
where sets $U$ and $V$ determine viable isolations of $d_1$ and $d_2$, and coefficients $a_U, b_V$ are determined. So since $f$ is a homomorphism, 
\begin{equation} \label{bracketproduct}
\bar{d_1}\bar{d_2} = \sum_{d\leq  d_1d_2} c_{d_1, d_2}^{d} [d].
\end{equation}
Since $QP_k(n)$ is closed under composition, we can also expand $\bar{d_1}\bar{d_2} $ in the basis $\{ \bar d ~|~ d \in \cD\}$. 
Again by Lemma \ref{lem:DominantTerms1}, for each $d \in \cD$, $[d]$ appears with non-zero coefficient in the bracket expansions of the elements of this basis: it appears with coefficient 1 in $\bar d$ and with coefficient 0 in $\bar d'$ for all $d \neq d' \in \cD$. Therefore, 
$$\bar{d_1}\bar{d_2} = \sum_{d\leq  d_1d_2 \atop 
d \in \cD} c_{d_1, d_2}^{d}\bar d$$
(where $c_{d_1, d_2}^{d}$ takes the same value as in \eqref{bracketproduct}). 
\end{proof}

\subsection{The generic quasi-partition algebra}
By Lemma \ref{lem:DominantTerms2} and the multiplication rules for $P_k(n-1)$, the coefficients $c_{d_1, d_2}^{d}=c_{d_1, d_2}^{d}(n)$ are well-defined rational functions of $n$ (with poles only at 0). Now fix a non-zero indeterminant $x$. Using this multiplication determined in Corollary \eqref{cor:barProduct}, define the \emph{general quasi-partition algebra} $QP_k(x)$ formally as 
$$QP_k(x) = \C(x)\text{-span}\,\cD \qquad \text{ with multiplication } \quad  d_1  d_2 = \sum_{d \in \cD \atop d\leq  d_1d_2} c_{d_1, d_2}^{d}(x)  d.$$

Note that when you specialize $x$ to an integer greater than $2k+1$, $QP_k(x)$ is isomorphic both the the centralizer of $S_x$ in $\End(W^{\otimes k})$ and to a subalgebra of $P_k(x-1)$.

%%%%%%%%%%%%%%%%%--- SECTION 3 -- %%%%%%%%%%%%%%%%%%%%%%%%%%%%%%%%

\section{Generators and Relations}

In this section we give a set of generators for $QP_k(n)$.  We first give a set of generators for the partition diagrams that do not have singleton vertices.  We then proceed to show that the image of $\pi^{\otimes k}$ of these generators form a generating set for $QP_k(n)$. 

It was shown in \cite[Lemma 3.1] {Tanabe}
 and \cite[Lemma 5.2]{Orellana}
 that the generators $ s_i, e_i,$ and  $b_i,$  for $i=1, \dots, k-1,$ generate those diagrams where all blocks contain an even number of nodes. The quasi-partition algebra additionally contains diagrams which has blocks with an odd number of vertices (and therefore an even number of odd blocks). 
To generate these additional diagrams, we will also need the generators 
$$\begin{tikzpicture}[yscale=.5, xscale=.4]
	\node[left] at (-2,.5) {$t_i = b_i p_{i+1} b_{i+1} = $};
	 \foreach \x in {-2,0,1,...,4,6} {
	 \filldraw [black] (\x, 0) circle (2pt);
	 \filldraw [black] (\x, 1) circle (2pt);
	 }
	\Nodes[0][4];
	\draw (1,0)--(1,1)--(2,1)-- (1,0);
	\draw (2,0)--(3,0)--(3,1)--(2,0) ;
	\draw (4,0)--(4,1) (0,1)--(0,0) (-2,1)--(-2,0) (6,1)--(6,0);
	\node at (-1,.5) {$\dots$};
	\node at (5,.5) {$\dots$};
	\node at (1,1.5) {$i$};
\ShiftX{19cm}
	\node[left] at (-2,.5) { and  \quad $h_i = b_i e_{i+1} p_i b_i = $};
	 \foreach \x in {-2,0,1,...,4,6} {
	 \filldraw [black] (\x, 0) circle (2pt);
	 \filldraw [black] (\x, 1) circle (2pt);
	 }
	\Nodes[0][4];
	\draw (1,1)--(2,1)--(3,1);
	\draw (1,0)--(2,0)--(3,0);
	\draw (4,0)--(4,1) (0,1)--(0,0) (-2,1)--(-2,0) (6,1)--(6,0);
	\node at (-1,.5) {$\dots$};
	\node at (5,.5) {$\dots$};
	\node at (1,1.5) {$i$};
	\node at (6.5,.5) {.};
	\end{tikzpicture} 
$$

In fact, we only need $h_1$ when $k=3$, since we have that $h_1 = t_2 e_1 (s_2 s_3 t_2 s_2 s_3)$ for  $k\geq 4$.

\begin{theorem}\label{thm:gens}
In $P_k(n)$, all diagrams in $\cD$ (i.e. all diagrams without isolated vertices) are generated by 
\begin{equation}\label{eq:Gens}
\cG =  \{ s_1, \dots, s_{k-1}, e_1, b_1, t_1, h_1\}.
\end{equation}
\end{theorem}
\begin{proof}
Note that by conjugating $b_1$ or  $e_1$ by words in $\{s_i ~|~ i = 1, \dots, k-1\}$, we can readily generate $e_i$ and $b_i$ for $i = 1, \dots, k-1$, and therefore, by  \cite[Lemma 3.1] {Tanabe}
 and \cite[Lemma 5.2]{Orellana}, all diagrams with even-sized blocks. For any diagram $d$ with an odd-sized block $I$ (and therefore at least one more, $J$), we will build $d$ from a diagram with two fewer odd blocks with the distinguished generators. 
The generators  $t_1, s_1, \dots, s_{k-1}$ generate all set partitions of the form 
$$t_{i_1, i_2, i_3} = \{i_1, i_2, i_1'\}, \{i_3, i_2', i_3'\}, \{1,1'\}, \{2, 2'\}, \dots, \{k, k'\}$$
and the generators $h_1,  s_1, \dots, s_{k-1}$ 
 generate all set partitions of the form 
$$h_{i_1, i_2, i_3} = \{i_1, i_2, i_3\}, \{i_1', i_2', i_3'\}, \{1,1'\}, \{2, 2'\}, \dots, \{k, k'\}.$$

We now prove by induction on the number of pairs of blocks of odd size in a diagram $d$ that we can always write it as a product of elements in $\cG$.  Notice that if the number of pairs of odd blocks is zero, we can write it as a product of the generators.  If the number of pairs in $d$  is greater than zero, then we have two cases to consider. 

\noindent{\sl Case 1:  Let $I$ and $J$ be two blocks of odd size in $d$ and assume $I \subseteq [k]$ and $J \subseteq [k']$}.  

Since $d$ has no singletons, each set has at least 3 elements. So let $i_1, i_2, i_3 \in I$ and $j'_1, j'_2, j'_3 \in J$. Let $d'$ be the diagram with the same blocks as $d$ except that $I$ and $J$ in $d$ have been replaced by the following sets in $d'$:
$$I \setminus \{i_1, i_2, i_3\}, \quad J \setminus \{j'_1, j'_2, j'_3\} ,\quad \{i_1, j'_1\},\quad \{i_2, j'_2\},\quad \{i_3, j'_3\}.$$ 
Then $d'$ has two fewer sets of odd size than $d$ and hence by induction it is the product of elements in $\cG$ . The diagram 
$$h_{i_1, i_2, i_3} d'$$
is the diagram obtained from $d$ by replacing $I$ and $J$ by the sets 
$$I \setminus \{i_1, i_2, i_3\}, \quad J \setminus \{j'_1, j'_2, j'_3\} ,\quad  \{i_1, i_2, i_3\}, \{j'_1, j'_2, j'_3\}.$$ 
If $I$ or $J$ has more than 3 elements (i.e.\ $I \setminus \{i_1, i_2, i_3\}$ or $J \setminus \{j'_1, j'_2, j'_3\}$ are non-empty), then for example 
$$I \setminus \{i_1, i_2, i_3\} \quad \text{ and } \quad  \{i_1, i_2, i_3\}$$
can be joined by right-multiplication of the even-block diagram
$$\{i_1, i_4, i_1', i_4'\}, \{1,1'\}, \dots, \{k,k'\},$$
where $i_4 \in I \setminus \{i_1, i_2, i_3\}$.  Hence, $d$ is the product of elements in $\cG$. 

\medskip

\noindent{\sl Case 2: Both $I$ and  $J$ have at least one element in $[k]$, and one of them has at least two elements in $[k]$.} (Otherwise, the same statement is true for $[k']$, and a similar construction can be used.)

Assume that $J \cap [k]$ has at least two elements $j_1$ and $j_2$. Let $d'$ be a new diagram with the sets $J \setminus j_1$ and $I \cup \{j_1\}$ in place of $I$ and $J$ in $d$. Then $d'$ has two fewer sets of odd size than $d$, and if $i \in I \cap [k]$, then 
$$d = t_{j_2, j_1, i} d'
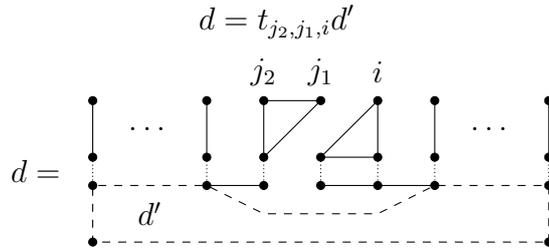
$$
$$\begin{tikzpicture}[scale=.75]
	\node at (-3,-.25) {$d=$};
	 \foreach \x in {-2,0,1,...,4,6} {
	 \filldraw [black] (\x, 0) circle (2pt);
	 \filldraw [black] (\x, 1) circle (2pt);
	 \draw[densely dotted] (\x,0)--(\x,-.5);
	 }
	\draw (1,0)--(1,1)--(2,1)-- (1,0);
	\draw (2,0)--(3,0)--(3,1)--(2,0) ;
	\draw (4,0)--(4,1) (0,1)--(0,0) (-2,1)--(-2,0) (6,1)--(6,0);
	\node at (-1,.5) {$\dots$};
	\node at (5,.5) {$\dots$};
	\node at (1,1.5) {$j_2$};
	\node at (2,1.5) {$j_1$};
	\node at (3,1.5) {$i$};
\pgftransformyshift{-1.5cm};
	 \filldraw [black] (-2, 0) circle (2pt);
	 \filldraw [black] (6, 0) circle (2pt);
	 \foreach \x in {-2,0,1,...,4,6} {
	 \filldraw [black] (\x, 1) circle (2pt);
	 }
	\draw (0,1)--(1,1) (2,1)--(3,1)--(4,1);
	\draw[dashed] (6,0)--(6,1)--(4,1)--(3,.5)--(1,.5)--(0,1)--(-2,1)--(-2,0)--(6,0) ;
	\node at (-1,.5) {$d'$};
	\end{tikzpicture} 
$$
\end{proof}

We have just shown that one can write $d=w$ as diagrams with $w = w_1 w_2 \cdots w_\ell$, $w_i \in \cG$. Further, the algorithm presented inductively on the number of pairs of odd blocks in $d$ provides a process for finding a $w$ satisfying further that for all $i=1, \dots, \ell-1$, the diagram $w_i w_{i+1} \cdots w_\ell$ also has no isolated vertices.

Corollary \ref{cor:barProduct} tells us that if $d_1, d_2 \in \cD$ and the diagram $d_1d_2$ has an isolated vertex, then $\bar d_1 \bar d_2 = 0$. However, if the diagram $d_1d_2$ does not have an isolated vertex, Corollary \ref{cor:barProduct} does not tell us which terms appear with non-zero coefficient. In particular, it is not immediately obvious that $c_{d_1, d_2}^{d_1d_2}$ is non-zero. A case-by-case calculation (done in Appendix \ref{app:TopTerms}) gives the following lemma. 

\begin{lemma} \label{lem:TopTerms}
If $d_1 \in \cG$ is one of the generators of $\cD$ and $d_2\in \cD$, then either 
\begin{enumerate}
\item $d_1 d_2 \in \cD$  and then coefficient $c_{d_1, d_2}^{d_1d_2}$ of $\overline{d_1 d_2}$ in $\bar d_1 \bar d_2$ is non-zero, or 
\item  $d_1 d_2 \notin \cD$, then $\bar d_1 \bar d_2 = 0$. 
\end{enumerate}
\end{lemma}

\begin{corollary}
As before, let $\bar d = \pi^{\otimes k} \circ d$. Then $QP_k(n)$ is generated by 
$\{ \bar d ~|~ d \in \cG \}$, where $\cG$ is as in \eqref{eq:Gens}.
\end{corollary}

\begin{proof}
Let $d \in \cD$. By Theorem \ref{thm:gens}, we can generate $d$ as a diagram from $\cG$; let $w = w_1 w_2 \cdots w_\ell$ be a word in the elements of $\cG$ so that $d = w$ as diagrams, satisfying for all $i=1, \dots, \ell-1$, the diagram $w_i w_{i+1} \cdots w_\ell \in \cD$ (as done immediately after the proof of Theorem \ref{thm:gens}). Using Lemma \ref{lem:TopTerms} inductively, $\bar w$ appears with non-zero coefficient in $\bar w_1 \cdots \bar w_\ell$. Using the triangularity of Corollary \ref{cor:barProduct}, we can then generate all of $\{\bar d ~|~ d \in \cD\}$. 
\end{proof}

We can use Lemmas \ref{lem:DominantTerms1} and \ref{lem:DominantTerms2} to calculate the expansion of these elements of $QP_k(n)$ in terms of maps $[d]$, and thus use relations in $P_k(n-1)$ to determine relations in $QP_k(n)$. One can easily check that the generators satisfy  the following relations in $QP_k(n)$: 
$$
\bar{s}_i^2 =1, \quad \bar{s}_{i}\bar{s}_{i+1}\bar{s}_{i}=\bar{s}_{i+1}\bar{s}_{i}\bar{s}_{i+1}, \quad \bar{s}_{i}\bar{s}_{j}=\bar{s}_{j}\bar{s}_{i} \quad \text{ if } |i-j|>1,$$
$$\bar{e}_i^2 = (n-1)\bar{e}_i, \quad \bar{e}_i\bar{e}_{i\pm1}\bar{e}_i = \bar{e}_i , \qquad \bar{b}_i^2 = \frac{n-2}{n}\bar{b}_i + \frac{1}{n^2}\bar{e}_i$$
$$\bar{s}_i\bar{b}_i=\bar{b}_i\bar{s}_i=\bar{b}_i \quad \text{ if } 1\leq i \leq n-2, \qquad \bar{s}_i\bar{t}_i = \bar{t}_i\bar{s}_{i+1}= \bar{t}_i \quad \text{ if } 1\leq i \leq n-2, \quad \text{ and }$$
$$\bar{e}_i \bar{t}_i = \bar{t}_i \bar{e}_{i+1} =0.$$

Comparing these relations to  those of the partition algebra $P_k(n-1)$ 
$${s}_i^2 =1, \quad {s}_{i}{s}_{i+1}{s}_{i}={s}_{i+1}{s}_{i}{s}_{i+1}, \quad {s}_{i}{s}_{j}={s}_{j}{s}_{i} \quad \text{ if } |i-j|>1,$$
$${e}_i^2 = (n-1){e}_i, \quad {e}_i{e}_{i\pm1}{e}_i = {e}_i , \qquad {b}_i^2 = {b}_i $$
$${s}_i{b}_i={b}_i{s}_i={b}_i \quad \text{ if } 1\leq i \leq n-2, \qquad {s}_i{t}_i = {t}_i{s}_{i+1}= {t}_i \quad \text{ if } 1\leq i \leq n-2, \quad \text{ and }$$
$${e}_i {t}_i = {t}_i {e}_{i+1} =0,$$
we observe that they are very similar but with additional lower terms in some cases.  For example, ${b}_i^2$ has additional terms, however, when $n\rightarrow \infty$, we have  $\bar{b}_i^2\rightarrow \bar{b}_i$.

%%%%%%%%%%%%%%%%%--- SECTION 4 -- %%%%%%%%%%%%%%%%%%%%%%%%%%%%%%%%

\section{Representation Theory of the Quasi-partition algebra}

In this section we describe the representation theory of $QP_k(n)$.  To do this we need to introduce partitions as the representation of both the symmetric group and $QP_k(n)$ is based on these combinatorial objects. 

Fix $n \in \Z_{\geq 0}$. A \emph{partition} $\lambda$ of $n$, denoted $\lambda\vdash n$, is a sequence of nonnegative integers 
$\lambda = (\lambda_1, \lambda_2, \ldots, \lambda_{\ell})$ such that $|\lambda|=\lambda_1+\lambda_2 + \cdots + \lambda_\ell =n$ and $\lambda_1\geq \lambda_2\geq \cdots \geq \lambda_\ell$.  If $n=0$, there is one partition, the empty partition, denoted by $\emptyset$.   
Given two partitions $\lambda$ and $\mu$ we say that $\mu\subseteq \lambda$ if $\mu_i\leq \lambda_i$ for all $i$.  

The irreducible representations $S^\lambda$ of $S_n$ are indexed by partitions $\lambda$ of $n$. As usual we identify a partition with its Young diagram, depicted as $|\alpha|$ boxes up and left justified, where the $i$th row has $\alpha_i$ boxes. In our setting, the combinatorics of the representation theory  of $QP_k(n)$ can be simplified by replacing partitions $\lambda \vdash n$ with partitions $(\lambda_2, \dots, \lambda_\ell)$ of $\lambda_2+\cdots+\lambda_\ell$. Thus the partitions of $n$ are in bijection with the partitions $\alpha$ of $m < n$ for which $\alpha_1 \leq n/2$.  For any such partition, let 
$$\VCenter{1}{$\bar{\alpha}= (n-|\alpha|, \alpha_1, \dots, \alpha_\ell). \qquad  
 \text{ For example, }$} 
\VCenter{1}{$(n-7, 3, 3, 1) = \overline{(3,3,1)}=$} \overline{\yng(3,3,1)} \VCenter{.9}{$.$}$$

In what follows we will need the following theorem. 
\begin{theorem}[The Centralizer Theorem]
Let $A$ be a finite dimensional algebra over $\mathbb{C}$. Let $M$ be a semi simple $A$-module and let $C=\End_A(M)$.  Suppose that $M\cong \bigoplus_{\lambda} (A^\lambda)^{\oplus m_\lambda}$, where $A^\lambda$ are irreducible $A$-modules and $m_\lambda\in \mathbb{Z}_{\geq 0}$ are multiplicities of $A^\lambda$ in the decomposition of $M$.   Then
\begin{enumerate}
\item[(a)] $C\cong \bigoplus_{\lambda} M_{m_\lambda}(\mathbb{C})$
\item[(b)] As an $(A, C)$-bimodule:  $M\cong \bigoplus_{\lambda} A^\lambda\otimes C^\lambda$, where $C^\lambda$ are simple $C$-modules. 
\end{enumerate}
\end{theorem}

By the  centralizer theorem, in order to understand the representation theory of $QP_k(n)$ we need to know how to decompose the $S_n$ module $W^{\otimes k}$.   The tensor product (or \emph{Kronecker product}) of two irreducible representations is usually not itself irreducible,  and a general rule for decomposing this tensor product is not known.   However, there are many results concerning stability of the product. For example, in \cite{BOR}, it is shown that if $n\geq |\alpha|+|\beta|+\alpha_1+\beta_1$ then the Kronecker product 
 \[S^{\bar\alpha}\otimes S^{\bar\beta} =\bigoplus_{\gamma}{g}_{\alpha\beta}^\gamma S^{\bar\gamma}\]
 is stable, meaning that for large $n$ the product does not depend on the first row of the partitions or equivalently it does not depend on $n$. 
In the special case when one of the representations is $S^{(n-1,1)} = S^{\overline{(1)}} = S^{\, \tiny\overline{\yng(1)}}$, we have the very well-known result 
\begin{equation}\label{one-box-decomp} S^{\overline{\alpha}} \otimes S^{\, \tiny\overline{\yng(1)}} = c(\alpha) (S^{\overline{\alpha}}) \oplus \bigoplus_{\beta \in \alpha^{\pm}}  S^{\overline{\beta}},\end{equation}
where $c(\alpha)$ is the number of corner boxes of $\alpha$ and $\alpha^{\pm}$ is the set of partitions $\beta$ with $\beta_1 \leq n/2$ gotten from $\alpha$ by (1) adding a box, (2) removing a box, or (3) moving a corner box of alpha to another corner. 
That is, $\beta$ differs from $\alpha$ from the position of a corner.   For example, if $n$ is large enough,
$$
S^{\tiny \overline{\yng(2,1)}} \otimes S^{\tiny \overline{\yng(1)}}  =  2 \  S^{\tiny \overline{\yng(2,1)}} \oplus S^{\tiny \overline{\yng(2)}} \oplus S^{\tiny \overline{\yng(1,1)}} \oplus S^{\tiny \overline{\yng(3)}} \oplus S^{\tiny \overline{\yng(1,1,1)}} \oplus    S^{\tiny \overline{\yng(3,1)}} \oplus S^{\tiny \overline{\yng(2,2)}} \oplus S^{\tiny \overline{\yng(2,1,1)}}.
$$

 If $n\gg 0$, then we get a stable product and we can ignore the first row of the partitions.

 \begin{theorem}  Let $k\geq 2$ and assume that $n\geq 2k$.  If $L_k(\lambda)$ denote the irreducible representations of $QP_k(n)$,  we have the decomposition of $W^{\otimes k}$ as an $(S_n,QP_k(n))$-bimodule
 \[W^{\otimes k} = \bigoplus S^{\overline{\lambda}} \otimes L_k(\lambda)\]
 where the sum is over all partitions $\overline{\lambda}$ of $n$ such that $|\lambda|\leq k$. 
 
 \end{theorem}
 \begin{proof}
Since $n\geq 2k$, we have that $QP_k(n)$ is isomorphic to the centralizer algebra.  Then by the rule for tensoring $S^{(n-1,1)}\otimes S^\lambda$ and the  centralizer theorem we get the decomposition. 
 \end{proof}
 
 We get as an immediate consequence a labeling set for the irreducible representations of $QP_k(n)$. 
 
 \begin{corollary} For $n/2 \geq k \geq 2$ the irreducible representations of $QP_k(n)$ are indexed by partitions $\mu$ of $0, 1, \ldots, k$.
\end{corollary}

%%%%%%%%%%%%%%%%%%%%%%%%%%%%%%
%%%%%%%%%%%%%%%%%%%%%%%%%%%%%%
%%%%%%%%%%%%%%%%%%%%%%%%%%%%%%

\subsection{Bratteli Diagram} We set $QP_0(n) =\mathbb{C}$.  And for any other $k\geq 1$ we have that 
\[QP_{k-1}(n) \subseteq QP_k(n)\]
We identify the elements in $QP_{k-1}(n)$ with the elements of $QP_{k}(n)$ that contain the block $\{ k, k'\}$. 
Hence, we have a tower of algebras
\begin{equation}\label{QPchain}
QP_0(n) \subset QP_1(n) \subset QP_2(n) \subset QP_3(n) \subset \cdots 
\end{equation}
For the remaining of the paper we assume that $n\gg 0$. 

Recall that we can represent the inclusion of $A\subset B$ of multimatrix algebras (with the same unit) by a bipartite graph.  The vertices in the graph are labeled by the simple summands of $A$ and $B$.  The number of edges joining a vertex $v$ for $A$ to a vertex $w$ for $B$ is the number of times the representation $v$ occurs in the restriction $w$ to $A$. 

In the case that we have a sequence of inclusions $A_0\subset A_1 \subset A_2 \subset \cdots$ of multimatrix algebras, one may connect the bipartite graphs describing the inclusions $A_i\subset A_{i+1}$, to obtain the \emph{Bratteli diagram}. 

We build the graph $\hat{P}$ as follows:

\begin{itemize} 

\item  vertices on level $k=0$: $\cP_0=\{\emptyset\}$, on level $k=1$, $\cP_1=\{(1)\}$ and for 
$k\geq 2$,  $\cP_k =\{\mu\, |\,  |\mu|\leq k\}$.
\item  edges $\lambda\rightarrow \mu$ for $\lambda \in \cP_{k-1}$, $\mu\in \cP_k$: 
\begin{itemize}
\item $\lambda$ is connected to $\mu$ by $c(\lambda)$ edges; 
\item $\lambda$ is connected to $\mu$ by one edge if $\lambda$ differs from $\mu$ by the position of one corner or by the removal or addition of one corner. 
\end{itemize}
\end{itemize}

\begin{proposition}
If we assume that $n\rightarrow \infty$.   The Bratteli diagram of the chain
\[QP_0(n) \subset QP_1(n) \subset QP_2(n) \subset QP_3(n)\subset \cdots \]
is the graph $\hat{P}$. 

\end{proposition}

\begin{proof}
Since $QP_k(n)$ is a centralizer algebra, we know by double centralizer theory that the decomposition of $W^{\otimes k}$ as an $S_n$-module yields the decomposition as a $QP_k(n)$-module. So we can use the decomposition rule in \eqref{one-box-decomp} to construct the Bratteli diagram for the chain in \eqref{QPchain} is a leveled graph completely described by  $\hat{P}$.

 The vertices at the $k$-th level of the Bratteli diagram are indexed by the irreducible components of $W^{\otimes k}$ as an $S_n$-module and the number of edges joining a vertex indexed by $\rho$ in the $(k-1)$-st  level to a vertex $\pi$ in the $k$-th is given by the multiplicity of $\pi$ in $V_\rho \otimes W$ as an $S_n$-module.

In other words, the vertices at the $i$th level index the irreducible components of $W^{\otimes k}$ as an $S_n$-module, and the number of edges joining $\alpha$ in the $i$th  level to a vertex $\beta$ in the $(i+1)$th level is the multiplicity of $S^{\overline{\beta}}$ in $S^{\overline{\alpha}} \otimes W$ as an $S_n$-module. 

\end{proof}

\vskip.1in
\begin{center}
\centerline{\bf Bratteli diagram for $QP_0(n) \subseteq QP_1(n) \subseteq QP_{2}(n) \subseteq \cdots$, levels 0--3.}
$$
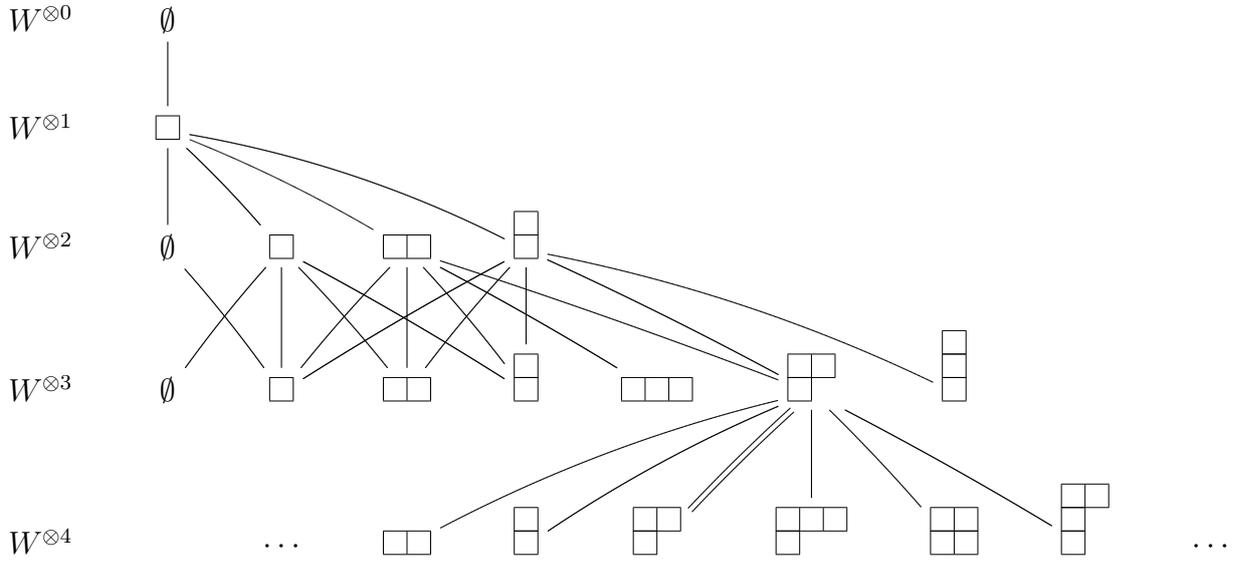
\xymatrix{
W^{\otimes 0} 
&\emptyset
\ar@{-}[d]
\\
W^{\otimes 1}
&\tiny{\yng(1)} 
		\ar@{-}[d]
		\ar@{-}@/^.1pc/[dr]
		\ar@{-}@/^.2pc/[drr]
		\ar@{-}@/^.5pc/[drrr]
\\
W^{\otimes 2}
&\emptyset 
		%\ar@{-}[d]
		\ar@{-}@/^.1pc/[dr]
& \tiny{\yng(1)} 
		\ar@{-}@/_.1pc/[dl]
		\ar@{-}[d]
		\ar@{-}@/^.1pc/[dr]
		\ar@{-}@/^.1pc/[drr]
& \tiny{\yng(2)} 
		\ar@{-}@/_.1pc/[dl]
		\ar@{-}[d]
		\ar@{-}@/^.1pc/[dr]
		\ar@{-}@/^.1pc/[drr]
		\ar@{-}@/^.1pc/[drrr]
& \tiny{\yng(1,1)}
		\ar@{-}@/_.1pc/[dll]
		\ar@{-}@/_.1pc/[dl]
		\ar@{-}[d]
		\ar@{-}@/^.1pc/[drr]
		\ar@{-}@/^.5pc/[drrr]\\
W^{\otimes 3}
&\emptyset & \tiny{\yng(1)} & \tiny{\yng(2)} & \tiny{\yng(1,1)} & \tiny{\yng(3)} & \tiny{\yng(2,1)} 
		\ar@{-}@/_.5pc/[dlll]
		\ar@{-}@/_.3pc/[dll]
		\ar@2{-}@/_.1pc/[dl]
		\ar@{-}[d]
		\ar@{-}@/^.1pc/[dr]
		\ar@{-}@/^.1pc/[drr]
& \tiny{\yng(1,1,1)}\\
W^{\otimes 4}&&\cdots&\tiny{\yng(2)} & \tiny{\yng(1,1)} & \tiny{\yng(2,1)} & \tiny{\yng(3,1)} & \tiny{\yng(2,2)} & \tiny{\yng(2,1,1)} &\cdots
}$$
\end{center}

\begin{remark}
Notice that in general, the Bratteli diagram is not multiplicity free; if $\lambda$ has more than one corner in its diagram, there are that number of edges from $\lambda$ in level $i$  to $\lambda$ in level $i+1$. 
\end{remark}

As usual, we get that the number of paths from $\emptyset$ to $\lambda$ is equal to the dimension of the irreducible representation of $QP_k(n)$ indexed by $\lambda$.  These paths can be encoded by the set of tableaux  in the next definition.

 \begin{defi} Given a positive integer $k$ and partition $\lambda$, a sequence $T=(\mu^0, \ldots, \mu^k )$ is called an {\em Kronecker tableau} of shape $\lambda$ if it is a sequence of $k$ Young diagrams such that $\mu^0 =\emptyset$ and $\mu^k =\lambda$ and, for every pair $\mu^i$ and $\mu^{i+1}$  of consecutive diagrams, either $\mu^{i+1}$ is obtained from $\mu^i$ by the addition or removal of one corner, or $\mu^{i+1}$  differs from $\mu^i$ by the position of one corner, or $\mu^{i+1}=\mu^i$ and $\mu^{i+1}$ has one distinguished corner.   
 \end{defi}
 
For example, if $k=9$ a possible Tableaux of shape $\lambda = {\tiny {\yng(2)}}$ is
\[ T=(\emptyset, {\tiny {\yng(1)}}, {\tiny {\young(X)}}, {\tiny {\yng(2)}}, {\tiny {\yng(1,1)}}, {\tiny {\yng(2,1)}}, {\tiny {\young(~~,X)}}, {\tiny {\yng(1,1,1)}}, {\tiny {\yng(2,1)}}, {\tiny {\yng(2)}}) \]
Here we have indicated the distinguished corner by an $X$.   

The following is a direct consequence of the double centralizer Theorem as the number of paths from $\emptyset$ to $\lambda$ is equal to the dimension of  $L(\lambda)$. 

\begin{lemma} \label{lemma:nkt}
Let $\lambda$ be a partition indexing an irreducible representation $QP_k(n)$.   Then the number of Kronecker tableaux of shape $\lambda$ of length $k$  is equal to the dimension of the $L(\lambda)$. 
\end{lemma}

In the following theorem we give an exact formula for these dimensions.

\begin{theorem} Let $k$ and $n$ be two positive integers and $\lambda$ a partition of $n$ such that $n> k+\lambda_2$. Then the 
dimension of the irreducible representation indexed by $\lambda$ in $QP_{k}(n)$ is  
\[ f^{\lambda}\sum_{m_1=0}^{|\lambda|}\left( {k\choose m_1} \sum_{m_2=|\lambda|-m_1}^{\lfloor \frac{k-m_1}{2}\rfloor} {m_2\choose |\lambda|-m_1} sp_2(k-m_1,m_2)\right), \]
where $f^{\lambda}$ is the number of standard tableaux of shape $\lambda$ and $sp_2(a,b)$ is the number of set partitions of a set with $a$ elements into $b$ parts of size at least 2. 
\end{theorem}
\begin{proof}
In \cite[Prop.\ 2]{CG} Chauve an Goupil have counted the number of Kronecker tableaux, hence by Lemma \ref{lemma:nkt} the result follows. 
\end{proof}

The number of paths from $\emptyset$ to $\lambda$ and back to $\emptyset$ is equal to the square of the dimension of $\lambda$. If we sum these values we obtain the dimension of $QP_k(n)$.  In our case, we get for the sum of squares of the dimensions of the irreducible representations is $1, 1, 4, 41, 715, 17722, \ldots$ (A000296) is the number of partitions of a $2k$-set into blocks of size  greater than 1, as expected.

\appendix 

\section{Finding the non-zero ``top terms''}\label{app:TopTerms}

\def\NodeA[#1,#2]{\filldraw [black] (#1, #2) circle (2pt);
				\draw (#1,#2-.4)--(#1,#2)--(#1+.4,#2-.4);}
\def\NodeB[#1,#2]{\draw (#1, #2) circle (4pt);
		\filldraw [black] (#1, #2) circle (1.5pt);}
\def\FrameA{
	\foreach \x in {3,4,5}{\draw (\x,1)--(\x,2);}
	\draw (.75,.75) -- (5.25,.75)--(5.25,-.5)--(.75,-.5)--(.75,.75);
	\draw[very thick, dotted] (1,1)--(1,.5) (2,1)--(2,.5);
	\foreach \x in {3,4,5}{\draw[very thick, dotted] (\x,1)--(\x,.75);}
	\Nodes[1][5]\Nodes[2][5]
	}
\def\FrameB{
	\foreach \x in {4,5}{\draw (\x,1)--(\x,2);}
	\draw (.75,.75) -- (5.25,.75)--(5.25,-.5)--(.75,-.5)--(.75,.75);
	\foreach \x in {1,2,3}{\draw[very thick, dotted] (\x,1)--(\x,.5);}
	\foreach \x in {3,4,5}{\draw[very thick, dotted] (\x,1)--(\x,.75);}
	\Nodes[1][5]\Nodes[2][5]
	}

\def\UNIT{.6}

In this appendix we prove Lemma \ref{lem:TopTerms}.  To do this we would like to track the coefficient of $[d_1d_2]$ (and therefore the coefficient of $\overline{d_1d_2}$) in the expansion of $\bar d_1 \bar d_2$ according to Lemmas \ref{lem:DominantTerms1} and \ref{lem:DominantTerms2}. If $[(d_1)_X][(d_2)_Y] = [d_1][d_2]$ as diagrams (neglecting coefficients), then necessarily, $X \subseteq [k']$ and $Y \subseteq [k]$. Therefore, we will restrict to focusing on the products of those terms 
\begin{enumerate}
\item in the expansion of $\bar d_1$ which isolate only bottom vertices, and 
\item in the expansion of $\bar d_2$ which isolate only top vertices.
\end{enumerate}

Further, if $d_1$ has a vertical size-2 block containing $i'$, then $Y$ cannot contain $i$ (and vice-versa). 
Note that for $i = 1, \dots, k-1$, since $\bar s_i = [s_i]$, we have $\bar s_i \bar d = \overline{s_i d}$. We continue with calculations for $d_1 = e_1, b_1, t_1,$ and $h_1$, and let $d_2 = d$ be any diagram in $\cD$.

\subsection{The coefficient of $[e_1d]$ in $\bar e_1 \bar d$.} The expansion of $\bar e_1$ is 
$$\bar{e}_1 = [e_1] + [(e_1)_{\{1',2' \}}] -\frac{1}{n}[(e_1)_{\{1,2 \}}]-\frac{1}{n}[(e_1)_{\{1,2,1',2'\}}].$$
The terms $[(e_1)_{X}] $ for which $X \subseteq [k']$ are $[e_1] + [(e_1)_{\{1',2' \}}]$. The only vertices which can then be isolated in $d$ are $1$ and $2$. Thus, possible contributions to $[e_1d]$ are
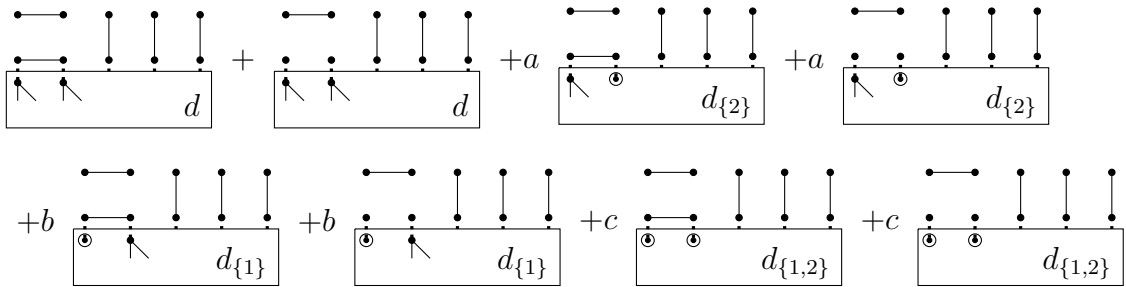
\begin{align*}
&~\begin{tikzpicture}[scale=\UNIT]
\draw (1,1)--(2,1) (1,2)--(2,2);
\NodeA[1,.5]\NodeA[2,.5]
\node[left] at (5.25,0){$d$};
\FrameA
\end{tikzpicture}
~\VCenter{3*\UNIT}{$+$}~
\begin{tikzpicture}[scale=\UNIT]
\draw  (1,2)--(2,2);
\NodeA[1,.5]\NodeA[2,.5]
\node[left] at (5.25,0){$d$};
\FrameA
\end{tikzpicture}
~\VCenter{3*\UNIT}{$+a$}~
\begin{tikzpicture}[scale=\UNIT]
\draw (1,1)--(2,1) (1,2)--(2,2);
\NodeA[1,.5]\NodeB[2,.5]
\node[left] at (5.25,0){$d_{\{2\}}$};
\FrameA
\end{tikzpicture}
~\VCenter{3*\UNIT}{$+a$}~
\begin{tikzpicture}[scale=\UNIT]
\draw (1,2)--(2,2);
\NodeA[1,.5]\NodeB[2,.5]
\node[left] at (5.25,0){$d_{\{2\}}$};
\FrameA
\end{tikzpicture}
\\
&~\VCenter{3*\UNIT}{$+b$}~
\begin{tikzpicture}[scale=\UNIT]
\draw (1,1)--(2,1) (1,2)--(2,2);
\NodeB[1,.5]\NodeA[2,.5]
\node[left] at (5.25,0){$d_{\{1\}}$};
\FrameA
\end{tikzpicture}
~\VCenter{3*\UNIT}{$+b$}~
\begin{tikzpicture}[scale=\UNIT]
\draw (1,2)--(2,2);
\NodeB[1,.5]\NodeA[2,.5]
\node[left] at (5.25,0){$d_{\{1\}}$};
\FrameA
\end{tikzpicture}
~ \VCenter{3*\UNIT}{$+ c$}~
\begin{tikzpicture}[scale=\UNIT]
\draw (1,1)--(2,1) (1,2)--(2,2);
\NodeB[1,.5]\NodeB[2,.5]
\node[left] at (5.25,0){$d_{\{1,2\}}$};
\FrameA
\end{tikzpicture}
~\VCenter{3*\UNIT}{$+c$}~
\begin{tikzpicture}[scale=\UNIT]
\draw  (1,2)--(2,2);
\NodeB[1,.5]\NodeB[2,.5]
\node[left] at (5.25,0){$d_{\{1,2\}}$};
\FrameA
\end{tikzpicture}
\\
\end{align*}
In the diagrams we emphasize that a vertex is isolated by circling it. 

\begin{enumerate}[C{a}se 1:]
\item  If the vertices $1$ and $2$ are in separate blocks, then $[e_1d]$ only appears in one place, since right multiplication of $d$ by $e_1$ results in $1$ and $2$ being in the same block.  The coefficient of $[e_1d]$ is 1. 

%$e_1$ joins their blocks together. 

\item The vertices $1$ and $2$ are in the same block.
\begin{enumerate}
\item If $\{1,2\}$ is a block: 
\begin{tikzpicture}[scale=.4]
\draw (.75,.75) -- (5.25,.75)--(5.25,-.5)--(.75,-.5)--(.75,.75);
\node[left] at (5.25,0){$d$};
\Nodes[.5][2]
\draw (1,.5)--(2,.5);
\end{tikzpicture}
\\
Then $d_{\{1\}} = d_{\{2\}} = d_{\{1,2\}}$, so $a = b = -1/n$ and $c = 1/n$. Then the coefficient on $[e_1d]$ is 
$$(n-1) +(n-1) - (1/n)((n-1) + (n-1)^2)=\fbox{$n-1$}.$$

\item If $\{1, 2, m'\}$ or $\{1, 2, m\}$ is a block, then
$e_1 d $ has an isolated vertex, so $\bar e_1 \bar d = 0$.

\item The vertices $1$ and $2$ are in the same block, and are connected to at least two other vertices.\\
Then $a = b = -1/n$, $c = 1/n^2$, and the coefficient on $[e_1d]$ is 
$$1 + 1 -(1/n)(1+(n-1) + 1 + (n-1)) + (1/n^2)((n-1) + (n-1)^2) = \fbox{$\frac{n-1}{n}$}.$$
\end{enumerate}

\end{enumerate}

\subsection{The coefficient of $[b_1d]$ in $\bar b_1 \bar d$.} The expansion of $\bar b_1$ only has one term where the isolation avoids top vertices, namely $[b_1]$. Again, this restricts us to terms in $\bar d$ where the isolation is restricted to the first two vertices. Thus the possible contributions to $[b_1d]$ are
$$
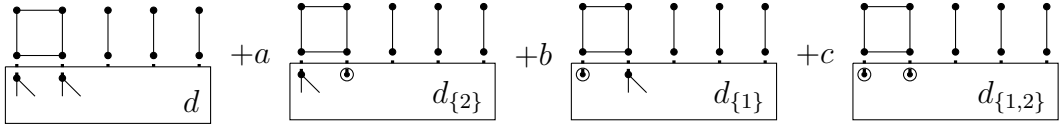
\begin{tikzpicture}[scale=\UNIT]
\draw (1,1)--(2,1)--(2,2)--(1,2)--(1,1);
\NodeA[1,.5]\NodeA[2,.5]
\node[left] at (5.25,0){$d$};
\FrameA
\end{tikzpicture}
~\VCenter{3*\UNIT}{$+a$}~
\begin{tikzpicture}[scale=\UNIT]
\draw (1,1)--(2,1)--(2,2)--(1,2)--(1,1);
\NodeA[1,.5]\NodeB[2,.5]
\node[left] at (5.25,0){$d_{\{2\}}$};
\FrameA
\end{tikzpicture}
~\VCenter{3*\UNIT}{$+b$}~
\begin{tikzpicture}[scale=\UNIT]
\draw (1,1)--(2,1)--(2,2)--(1,2)--(1,1);
\NodeB[1,.5]\NodeA[2,.5]
\node[left] at (5.25,0){$d_{\{1\}}$};
\FrameA
\end{tikzpicture}
~ \VCenter{3*\UNIT}{$+ c$}~
\begin{tikzpicture}[scale=\UNIT]
\draw (1,1)--(2,1)--(2,2)--(1,2)--(1,1);
\NodeB[1,.5]\NodeB[2,.5]
\node[left] at (5.25,0){$d_{\{1,2\}}$};
\FrameA
\end{tikzpicture}
$$

\begin{enumerate}[C{a}se 1:]
\item  The vertices $1$ and $2$ are in separate blocks.\\
In this case, $[b_1d]$ only appears in one place, since $b_1$ joins their blocks together.  Then the coefficient of $[b_1d]$ is 1.

\item The vertices $1$ and $2$ are in the same block.
\begin{enumerate}
\item 
If $\{1,2\}$ is a block, then $d_{\{1\}} = d_{\{2\}} = d_{\{1,2\}}$. So $a = b = -1/n$ and $c = 1/n$, and the coefficient on $[b_1d]$ is 
$1-1/n-1/n+1/n= \fbox{$\frac{n-1}{n}$}.$

\item If $\{1, 2, m'\}$ or $\{1, 2, m\}$ is a block:\\
Then $d_{\{1,2\}}$ doesn't contribute, and $a = b = -1/n$. Since $b_1$ joins $1$ and $2$ back together, $[b_1d]$ has coefficient 
$1 -1/n - 1/n = \fbox{$\frac{n-2}{n}$}.$

\item The vertices $1$ and $2$ are in the same block, and are connected to at least two other vertices.\\
Then $a = b = -1/n$, $c = 1/n^2$, and the coefficient on $[b_1d]$ is 
$$1 -1/n-1/n+1/n^2= \fbox{$ \frac{(n-1)^2}{n^2}$}.$$
\end{enumerate}
\end{enumerate}

\subsection{The coefficient of $[t_1d]$ in $\bar t_1 \bar d$.} 
The expansion of $\bar t_1$ only has one term where the isolation avoids top vertices, namely $[t_1]$. This restricts us to terms in $\bar d$ where the isolation is in the first three vertices. However, we can also note that the terms 
$$
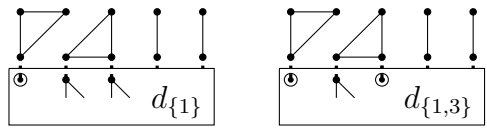
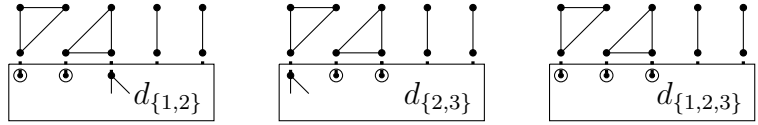
\begin{tikzpicture}[scale=\UNIT]
\draw (1,1)--(1,2)--(2,2)--(1,1) (2,1)--(3,1)--(3,2)--(2,1);
\NodeB[1,.5]\NodeA[2,.5]\NodeA[3,.5]
\node[left] at (5.25,0){$d_{\{1\}}$};
\FrameB
\end{tikzpicture}
\qquad 
\begin{tikzpicture}[scale=\UNIT]
\draw (1,1)--(1,2)--(2,2)--(1,1) (2,1)--(3,1)--(3,2)--(2,1);
\NodeB[1,.5]\NodeA[2,.5]\NodeB[3,.5]
\node[left] at (5.25,0){$d_{\{1,3\}}$};
\FrameB
\end{tikzpicture}
$$
$$\begin{tikzpicture}[scale=\UNIT]
\draw (1,1)--(1,2)--(2,2)--(1,1) (2,1)--(3,1)--(3,2)--(2,1);
\NodeB[1,.5]\NodeB[2,.5]\NodeA[3,.5]
\node[left] at (5.25,0){$d_{\{1,2\}}$};
\FrameB
\end{tikzpicture}
\qquad 
\begin{tikzpicture}[scale=\UNIT]
\draw (1,1)--(1,2)--(2,2)--(1,1) (2,1)--(3,1)--(3,2)--(2,1);
\NodeA[1,.5]\NodeB[2,.5]\NodeB[3,.5]
\node[left] at (5.25,0){$d_{\{2,3\}}$};
\FrameB
\end{tikzpicture}
\qquad 
\begin{tikzpicture}[scale=\UNIT]
\draw (1,1)--(1,2)--(2,2)--(1,1) (2,1)--(3,1)--(3,2)--(2,1);
\NodeB[1,.5]\NodeB[2,.5]\NodeB[3,.5]
\node[left] at (5.25,0){$d_{\{1,2,3\}}$};
\FrameB
\end{tikzpicture}
$$
will not contribute to $[t_1d]$ in any case. Therefore the possible contributors are
$$\quad
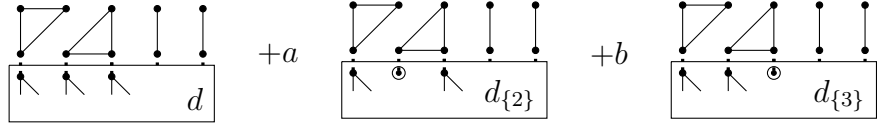
\begin{tikzpicture}[scale=\UNIT]
\draw (1,1)--(1,2)--(2,2)--(1,1) (2,1)--(3,1)--(3,2)--(2,1);
\NodeA[1,.5]\NodeA[2,.5]\NodeA[3,.5]
\node[left] at (5.25,0){$d$};
\FrameB
\end{tikzpicture}
\quad\VCenter{3*\UNIT}{$+a$}\quad
\begin{tikzpicture}[scale=\UNIT]
\draw (1,1)--(1,2)--(2,2)--(1,1) (2,1)--(3,1)--(3,2)--(2,1);
\NodeA[1,.5]\NodeB[2,.5]\NodeA[3,.5]
\node[left] at (5.25,0){$d_{\{2\}}$};
\FrameB
\end{tikzpicture}\quad \VCenter{3*\UNIT}{$+ b$}\quad
\begin{tikzpicture}[scale=\UNIT]
\draw (1,1)--(1,2)--(2,2)--(1,1) (2,1)--(3,1)--(3,2)--(2,1);
\NodeA[1,.5]\NodeA[2,.5]\NodeB[3,.5]
\node[left] at (5.25,0){$d_{\{3\}}$};
\FrameB
\end{tikzpicture}
$$

\begin{enumerate}[C{a}se 1:]
\item  If $2$ and $3$ are in separate blocks,
then $[t_1d]$ only appears in one place, since $t_1$ joins their blocks together.  In this case its coefficient is 1.

\item If $\{2,3\}$ is a block, 
then $\bar t_1 \bar d = 0$.

\item Otherwise, $a = b = -1/n$ and the coefficient on $[t_1d]$ is $1 - 2/n = \fbox{$\frac{n - 2}{n}$}$.

\end{enumerate}

\subsection{The coefficient of $[h_1d]$ in $\bar h_1 \bar d$.} 
The expansion of $\bar h_1$ has two terms which have no top vertices isolated, namely 
$ [h_1] - [(h_1)_{\{1',2',3'\}}].$ This restricts us to terms in $\bar d$ where the isolation is in the first three vertices, and the possible contributors to $[h_1d]$ are
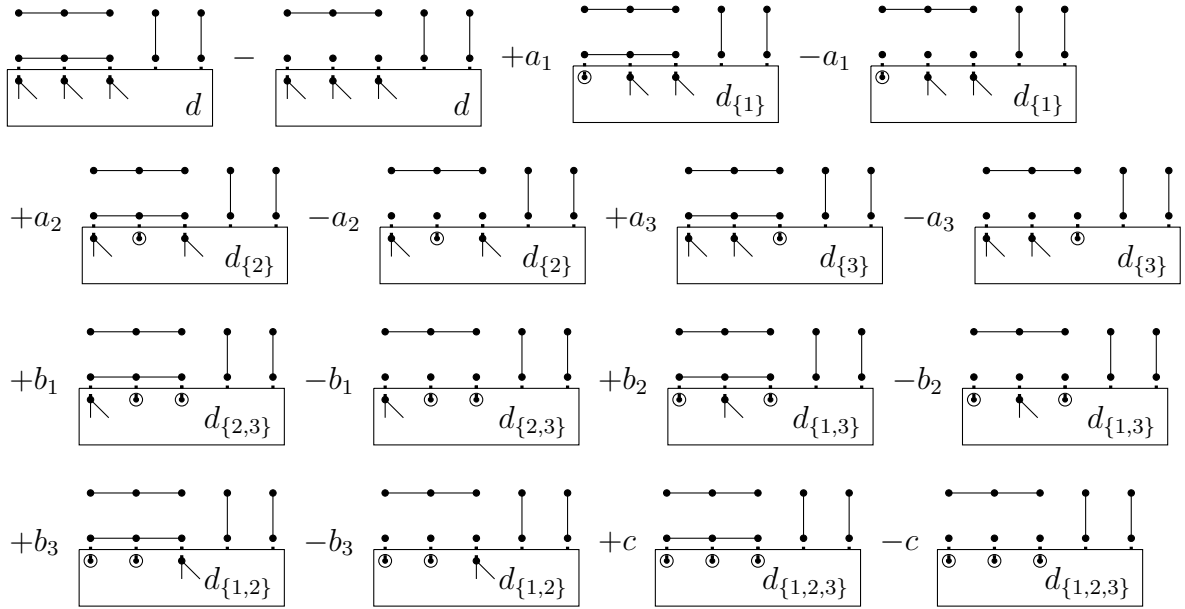
\def\UNIT{.6}
\begin{align*}
&~ \begin{tikzpicture}[scale=\UNIT]
\draw (1,1)--(2,1)--(3,1) (1,2)--(2,2)--(3,2);
\NodeA[1,.5]\NodeA[2,.5]\NodeA[3,.5]
\node[left] at (5.25,0){$d$};
\FrameB
\end{tikzpicture}
~ \VCenter{3*\UNIT}{$-$}~
\begin{tikzpicture}[scale=\UNIT]
\draw (1,2)--(2,2)--(3,2);
\NodeA[1,.5]\NodeA[2,.5]\NodeA[3,.5]
\node[left] at (5.25,0){$d$};
\FrameB
\end{tikzpicture}
~\VCenter{3*\UNIT}{$+a_1$}~
\begin{tikzpicture}[scale=\UNIT]
\draw (1,1)--(2,1)--(3,1) (1,2)--(2,2)--(3,2);
\NodeB[1,.5]\NodeA[2,.5]\NodeA[3,.5]
\node[left] at (5.25,0){$d_{\{1\}}$};
\FrameB
\end{tikzpicture}
~ \VCenter{3*\UNIT}{$-a_1$}~
\begin{tikzpicture}[scale=\UNIT]
\draw (1,2)--(2,2)--(3,2);
\NodeB[1,.5]\NodeA[2,.5]\NodeA[3,.5]
\node[left] at (5.25,0){$d_{\{1\}}$};
\FrameB
\end{tikzpicture}\\
&\VCenter{3*\UNIT}{$+a_2$}~
\begin{tikzpicture}[scale=\UNIT]
\draw (1,1)--(2,1)--(3,1) (1,2)--(2,2)--(3,2);
\NodeA[1,.5]\NodeB[2,.5]\NodeA[3,.5]
\node[left] at (5.25,0){$d_{\{2\}}$};
\FrameB
\end{tikzpicture}
~ \VCenter{3*\UNIT}{$-a_2$}~
\begin{tikzpicture}[scale=\UNIT]
\draw (1,2)--(2,2)--(3,2);
\NodeA[1,.5]\NodeB[2,.5]\NodeA[3,.5]
\node[left] at (5.25,0){$d_{\{2\}}$};
\FrameB
\end{tikzpicture}
~\VCenter{3*\UNIT}{$+a_3$}~
\begin{tikzpicture}[scale=\UNIT]
\draw (1,1)--(2,1)--(3,1) (1,2)--(2,2)--(3,2);
\NodeA[1,.5]\NodeA[2,.5]\NodeB[3,.5]
\node[left] at (5.25,0){$d_{\{3\}}$};
\FrameB
\end{tikzpicture}
~ \VCenter{3*\UNIT}{$-a_3$}~
\begin{tikzpicture}[scale=\UNIT]
\draw (1,2)--(2,2)--(3,2);
\NodeA[1,.5]\NodeA[2,.5]\NodeB[3,.5]
\node[left] at (5.25,0){$d_{\{3\}}$};
\FrameB
\end{tikzpicture}\\
&\VCenter{3*\UNIT}{$+b_1$}~
\begin{tikzpicture}[scale=\UNIT]
\draw (1,1)--(2,1)--(3,1) (1,2)--(2,2)--(3,2);
\NodeA[1,.5]\NodeB[2,.5]\NodeB[3,.5]
\node[left] at (5.25,0){$d_{\{2,3\}}$};
\FrameB
\end{tikzpicture}
~ \VCenter{3*\UNIT}{$-b_1$}~
\begin{tikzpicture}[scale=\UNIT]
\draw (1,2)--(2,2)--(3,2);
\NodeA[1,.5]\NodeB[2,.5]\NodeB[3,.5]
\node[left] at (5.25,0){$d_{\{2,3\}}$};
\FrameB
\end{tikzpicture}
~\VCenter{3*\UNIT}{$+b_2$}~
\begin{tikzpicture}[scale=\UNIT]
\draw (1,1)--(2,1)--(3,1) (1,2)--(2,2)--(3,2);
\NodeB[1,.5]\NodeA[2,.5]\NodeB[3,.5]
\node[left] at (5.25,0){$d_{\{1,3\}}$};
\FrameB
\end{tikzpicture}
~ \VCenter{3*\UNIT}{$-b_2$}~
\begin{tikzpicture}[scale=\UNIT]
\draw (1,2)--(2,2)--(3,2);
\NodeB[1,.5]\NodeA[2,.5]\NodeB[3,.5]
\node[left] at (5.25,0){$d_{\{1,3\}}$};
\FrameB
\end{tikzpicture}\\
&\VCenter{3*\UNIT}{$+b_3$}~
\begin{tikzpicture}[scale=\UNIT]
\draw (1,1)--(2,1)--(3,1) (1,2)--(2,2)--(3,2);
\NodeB[1,.5]\NodeB[2,.5]\NodeA[3,.5]
\node[left] at (5.25,0){$d_{\{1,2\}}$};
\FrameB
\end{tikzpicture}
~ \VCenter{3*\UNIT}{$-b_3$}~
\begin{tikzpicture}[scale=\UNIT]
\draw (1,2)--(2,2)--(3,2);
\NodeB[1,.5]\NodeB[2,.5]\NodeA[3,.5]
\node[left] at (5.25,0){$d_{\{1,2\}}$};
\FrameB
\end{tikzpicture}
~\VCenter{3*\UNIT}{$+c$}~
\begin{tikzpicture}[scale=\UNIT]
\draw (1,1)--(2,1)--(3,1) (1,2)--(2,2)--(3,2);
\NodeB[1,.5]\NodeB[2,.5]\NodeB[3,.5]
\node[left] at (5.25,0){$d_{\{1,2,3\}}$};
\FrameB
\end{tikzpicture}
~ \VCenter{3*\UNIT}{$-c$}~
\begin{tikzpicture}[scale=\UNIT]
\draw (1,2)--(2,2)--(3,2);
\NodeB[1,.5]\NodeB[2,.5]\NodeB[3,.5]
\node[left] at (5.25,0){$d_{\{1,2,3\}}$};
\FrameB
\end{tikzpicture}\\
\end{align*}

\begin{enumerate}[C{a}se 1:]
\item If $1$, $2$, and $3$ are all in separate blocks, then
there is no repetition, and the top term appears with coefficient 1. 
\item The vertices $1$ and $2$ appear in the same block, but separate from 3 (similarly for 1 and 3 or 2 and 3). \\
\begin{enumerate}
\item If $\{1,2\}$ is a block: \\
If 3 is also in a block of size 2, then $h_1 d$ has an isolated vertex, and $\bar h_1 \bar d = 0$.

If 3 is in a larger block, then since $a_1=a_2=a_3=-b_3 = -1/n$, $b_1 = b_2 = -c=1/n^2$,
 the coefficient is 
 $$1 - (n-1) -(2/n)(1-(n-1)^2) -(1/n)((n-1) - (n-1)^2) $$ $$+(2/n^2)((n-1)-(n-1)^3) +(1/n)(1-(n-1)^2) - (1/n^2)((n-1) - (n-1)^3) $$ $$= \fbox{$(n-1)(n-2)$}.$$

\item Otherwise, none of the terms with $[ (h_1)_{\{1',2',3'\}}]$ will contribute.
So the only possible contributors are 
$$[h_1][d], \quad [h_1][d_{\{1\}}], \quad \text{ and } \quad [h_1][d_{\{2\}}]$$
(because we need $h_1$ to bond 3 to $1$ and $2$, and we need one of 1 or 2 to bond to the rest of their block). Then since $a_1 = a_2 = -1/n$, the top term has a coefficient of $1 - 2/n = \fbox{$\frac{n - 2}{n}$}$.
\end{enumerate}
\item The vertices $1$, $2$, and $3$ are all in the same block.
\begin{enumerate}
\item If $\{1,2,3\}$ is a block:\\
Then $d_{\{1,2\}} = d_{\{1,3\}} = d_{\{2,3\}} = d_{\{1,2,3\}}$, $b_1=b_2=b_3 = -c = 1/n^2$,\\
 $a_1 = a_2 = a_3=-1/n$. The coefficient on the top terms is  
$$(n-1)\left( 
1-1-(3/n)(1-(n-1)) + (2/n^2)(1-(n-1)^2) \right)=  \fbox{$\frac{(n-2)(n-1)}{n}$}
$$
\item If $\{1,2,3,m'\}$ or $\{1,2,3,m\}$ is a block,  then $\bar h_1 \bar d = 0$.
\item Otherwise, the vertices $1$, $2$, and $3$ are in a block of size 5 or more.\\
Then $a_1 = a_2 = a_3 = -1/n$, $b_1= b_2 = b_3 = 1/n^2$, and $c = -1/n^3$ and the coefficient on the top term is 
$$1-1 - (3/n)(1-(n-1) ) + (3/n^2)(1-(n-1)^2)-(1/n^3)((n-1)-(n-1)^3) =  \fbox{$\frac{ (n-1)(n-2)}{n^2}$}.$$

\end{enumerate}

\end{enumerate}

\end{document}